\documentclass[10pt]{article}

\usepackage[paperwidth=6in,paperheight=9in,left=0.5in,right=0.5in,top=0.65in,bottom=0.65in]{geometry}
\usepackage{amsmath,amssymb,amsthm,mathtools,bm}
\usepackage{newtxtext,newtxmath}
\usepackage{graphicx}
\usepackage{cite}
\usepackage{booktabs,array,multirow}
\usepackage{microtype}
\usepackage{enumitem}
\usepackage{hyperref}
\usepackage{aliascnt}
\usepackage{cleveref}
\usepackage{xcolor}
\usepackage{placeins}

\hypersetup{
  colorlinks=true,
  linkcolor=blue!55!black,
  citecolor=blue!55!black,
  urlcolor=blue!55!black,
  pdftitle={High-Dimensional Spectral Limits for Gaussian KL-Unbalanced Optimal Transport},
  pdfauthor={Jiaping Yang and Yunxin Zhang}
}
\numberwithin{equation}{section}

\newtheorem{theorem}{Theorem}[section]
\newaliascnt{proposition}{theorem}
\newtheorem{proposition}[proposition]{Proposition}
\aliascntresetthe{proposition}
\newaliascnt{lemma}{theorem}
\newtheorem{lemma}[lemma]{Lemma}
\aliascntresetthe{lemma}
\newaliascnt{corollary}{theorem}
\newtheorem{corollary}[corollary]{Corollary}
\aliascntresetthe{corollary}
\newaliascnt{assumption}{theorem}
\newtheorem{assumption}[assumption]{Assumption}
\aliascntresetthe{assumption}
\theoremstyle{definition}
\newaliascnt{definition}{theorem}
\newtheorem{definition}[definition]{Definition}
\aliascntresetthe{definition}
\newaliascnt{remark}{theorem}
\newtheorem{remark}[remark]{Remark}
\aliascntresetthe{remark}

\crefname{theorem}{Theorem}{Theorems}
\Crefname{theorem}{Theorem}{Theorems}
\crefname{proposition}{Proposition}{Propositions}
\Crefname{proposition}{Proposition}{Propositions}
\crefname{lemma}{Lemma}{Lemmas}
\Crefname{lemma}{Lemma}{Lemmas}
\crefname{corollary}{Corollary}{Corollaries}
\Crefname{corollary}{Corollary}{Corollaries}
\crefname{assumption}{Assumption}{Assumptions}
\Crefname{assumption}{Assumption}{Assumptions}
\crefname{definition}{Definition}{Definitions}
\Crefname{definition}{Definition}{Definitions}
\crefname{remark}{Remark}{Remarks}
\Crefname{remark}{Remark}{Remarks}

\newcommand{\R}{\mathbb R}

\newcommand{\tr}{\operatorname{tr}}
\newcommand{\KL}{\operatorname{KL}}
\newcommand{\diag}{\operatorname{diag}}
\newcommand{\spec}{\operatorname{spec}}

\newcommand{\Acal}{\mathcal A}
\newcommand{\Ucal}{\mathcal U}
\newcommand{\N}{\mathcal N}

\newcommand{\op}{\mathrm{op}}
\newcommand{\as}{\mathrm{a.s.}}
\newcommand{\ip}[2]{\left\langle #1,#2\right\rangle_F}
\newcommand{\normF}[1]{\left\|#1\right\|_F}
\newcommand{\normop}[1]{\left\|#1\right\|_{\op}}
\newcommand{\dd}{\,\mathrm d}

\begin{document}
\begin{center}
{\bfseries\large High-Dimensional Spectral Limits for Gaussian KL-Unbalanced Optimal Transport\par}
\vspace{8pt}
{\small Jiaping Yang \quad and \quad Yunxin Zhang\par}
{\small School of Mathematical Sciences, Fudan University, Shanghai 200433, China. \par}
\end{center}
\vspace{5pt}

\begin{abstract}
\small
We study high-dimensional random-matrix limits of Gaussian Kullback--Leibler unbalanced optimal transport (KL-UOT).  Under equal marginal penalties, the covariance action admits an exact log-determinant representation in terms of a nonlinear ridge product, together with a positive-semidefinite extension that remains finite at arbitrary aspect ratios.  For independent real Wishart samples, strong asymptotic freeness gives the limiting free multiplicative convolution and almost-sure Hausdorff convergence of the ridge-product spectrum; independent Haar orientations yield the corresponding first-order limit for deformed populations.  In the symmetric nonsingular identity-Wishart model, we derive an explicit $\eta$-transform and a low-degree algebraic equation that select the physical branch and determine the support interval, square-root edges, and extreme-eigenvalue limits.  We further obtain all-aspect one-sample Marchenko--Pastur limits under finite fourth moments, real-Gaussian Bai--Silverstein fluctuations for $c<1$, and a joint random-matrix/penalty limit showing that sample-covariance noise produces the critical scale $\tau_p\asymp p$.
\end{abstract}

{\small\noindent\textbf{Keywords.} linear spectral statistics; sample covariance matrices; free multiplicative convolution; Gaussian unbalanced optimal transport; Marchenko--Pastur law.\par}


\section{Introduction}

Spectral functionals of sample covariance matrices are central to random matrix theory and high-dimensional statistics.  When the dimension $p$ is comparable to the sample size $n$, individual sample eigenvalues no longer consistently estimate their population counterparts.  Their collective behavior is instead described by Marchenko--Pastur limits, deterministic equivalents, and central limit theorems for linear spectral statistics (LSS)~\cite{MarchenkoPastur1967,SilversteinChoi1995,BaiSilverstein2004,BaiSilverstein2010}.  These tools underlie methods for covariance estimation, signal processing, and high-dimensional testing~\cite{CouilletDebbah2011,YaoZhengBai2015,HachemLoubatonNajim2007}, including recent analyses of rescaled sample correlation matrices and structured covariance tests~\cite{ChenZhengZou2026,WangEtAl2026}.

This paper studies the random-matrix asymptotics of the covariance functional induced by Gaussian Kullback--Leibler unbalanced optimal transport (KL-UOT).  For Gaussian probability measures $\mu_i=\N(m_i,\Sigma_i)$, quadratic balanced transport has the closed form
\begin{equation}\label{eq:gaussian-w2}
W_2^2(\mu_0,\mu_1)
=\|m_0-m_1\|^2+\tr\Sigma_0+\tr\Sigma_1
-2\tr\!\left(\Sigma_1^{1/2}\Sigma_0\Sigma_1^{1/2}\right)^{1/2}.
\end{equation}
Its covariance term is the Bures--Wasserstein distance~\cite{Gelbrich1990,Takatsu2011,BhatiaJainLim2019}; the associated product spectral structure has already led to random-matrix corrections for high-dimensional Gaussian Wasserstein estimation~\cite{TiomokoCouillet2019}.  Related corrections are available for broader families of covariance-matrix distances and divergences~\cite{CouilletTiomokoZozorMoisan2019}, for asymptotic distances between sample covariance matrices~\cite{PereiraMestreGregoratti2024}, and for Fr\'echet means on the positive-definite cone~\cite{BouchardEtAl2024}.  These works provide the natural comparison point for our setting.  At fixed penalty, however, the KL-UOT functional studied below is not simply a function of $\Sigma_0^{-1}\Sigma_1$; it depends on a nonlinear ridge product that retains information about the relative eigenvectors in the two-sample regime.  This distinction is not merely a change of scalar test function.  Under equal penalties, KL-UOT first applies the nonlinear ridge map $h_r(x)=rx/(1+rx)$ to each covariance matrix and then couples the transformed matrices through $R_1^{1/2}R_0R_1^{1/2}$.  In the asymptotically free regime analyzed below, the two-sample limit is governed by the joint ridge-product law rather than by the two marginal sample spectra separately.

For finite measures $\alpha=a\mu_0$ and $\beta=b\mu_1$, we consider the KL-relaxed endpoint problem
\begin{equation}\label{eq:uot-def}
\Ucal_{\tau_0,\tau_1}(\alpha,\beta)
:=\inf_{\gamma\ge0}
\left\{
\int\|x-y\|^2\,\dd\gamma(x,y)
+\tau_0\KL(\gamma_0\mid\alpha)
+\tau_1\KL(\gamma_1\mid\beta)
\right\},
\end{equation}
Here $\tau_0,\tau_1>0$, and the coupling itself carries no entropy penalty; Section~\ref{sec:spectral} fixes the finite-measure KL convention.  Entropy-transport, dynamic, and Kantorovich formulations of unbalanced transport are developed in~\cite{LieroMielkeSavare2018,ChizatDynamic2018,ChizatEtAl2018,SejournePeyreVialard2023}.  Recent work has also addressed semi-dual estimation, stability, geometry, and minimax transport-growth estimation under quadratic cost and KL marginal penalties~\cite{VacherVialard2023,GallouetGhezziVialard2025,PonnopratIsobeImaizumi2026}.  Our question is different: we ask what happens when Gaussian covariance matrices are replaced by high-dimensional sample covariances, so that nonlinear spectral noise persists at leading order.

For Gaussian inputs, the endpoint problem without coupling entropy has an explicit finite-dimensional solution with adjusted Gaussian marginals and a matrix Riccati equation~\cite{YangZhang2026}.  Under equal penalties, this formulation coincides with the Gaussian Hellinger--Kantorovich (GHK) model in Janati's Gaussian endpoint terminology, not with the geodesic HK/Wasserstein--Fisher--Rao metric; Janati's thesis records the corresponding symmetric Gaussian value (Proposition~14, up to parameter normalization)~\cite{JanatiThesis2021}.  Related formulas with coupling entropy appear in~\cite{JanatiEtAl2020}, and an independent control-theoretic Gaussian reduction was developed in~\cite{NakashimaEtAl2026}.  Section~\ref{sec:spectral} recalls the finite-dimensional ingredients only to fix notation; derivations are given in Appendix~\ref{app:gaussian-derivations}.  The random-matrix contribution begins with the spectral reformulation in Theorem~\ref{thm:exact-spectral} and Corollary~\ref{cor:ridge-product}, followed by the proportional-dimensional, fluctuation, and free convolution analyses.

In proportional dimension, the deterministic Gaussian calculation becomes a random spectral problem.  Let
\begin{equation}\label{eq:intro-sample-cov}
S_{i,p}=\Sigma_{i,p}^{1/2}\frac{X_{i,p}^{\top}X_{i,p}}{n_i}\Sigma_{i,p}^{1/2},
\qquad \frac{p}{n_i}\longrightarrow c_i\in(0,\infty),
\end{equation}
and, in the equal-penalty case $r=2/\tau$, define the ridge transforms
\begin{equation}\label{eq:intro-ridge}
\widehat R_{i,p}=rS_{i,p}(I+rS_{i,p})^{-1}.
\end{equation}
The two-sample KL-UOT covariance action is governed by the noncommutative product
\begin{equation}\label{eq:intro-ridge-product}
\widehat R_{1,p}^{1/2}\widehat R_{0,p}\widehat R_{1,p}^{1/2}.
\end{equation}
The high-dimensional analysis reduces to nonlinear spectral functionals of $\widehat R_{i,p}$ and, in the two-sample problem, to the ridge product in \eqref{eq:intro-ridge-product}.  When $c_i<1$ these matrices are eventually positive definite and retain the original nondegenerate-Gaussian KL-UOT interpretation.  For $c_i\ge1$ we use the positive-semidefinite spectral extension introduced below; this keeps the random-matrix functional well defined without asserting a singular-Gaussian KL formula.  These objects place Gaussian KL-UOT within the settings of LSS, strong asymptotic freeness, and multiplicative free convolution~\cite{VoiculescuDykemaNica1992,NicaSpeicher2006,CollinsMale2014,BercoviciVoiculescu1993}.

The main results fall into four groups, with the two-sample ridge product as the central random-matrix object.
\begin{enumerate}[label=(\roman*),leftmargin=2.2em]
\item \textit{Spectral reformulation and positive-semidefinite extension.}  The finite-dimensional Gaussian endpoint solution is rewritten as an exact log-determinant spectral formula.  Under equal penalties it reduces to a ridge-product identity, which extends continuously to positive-semidefinite covariance matrices and therefore remains meaningful in rank-deficient proportional regimes.
\item \textit{Two-sample ridge-product laws and KL-UOT-specific edge identification.}  For independent identity-Wishart samples at arbitrary aspect ratios $c_i>0$, strong asymptotic freeness gives the free multiplicative-convolution law and Hausdorff spectral convergence.  Independent Haar orientations give the corresponding first-order limit for deformed populations.  In the symmetric nonsingular case, an explicit $\eta$-transform and low-degree algebraic equation select the physical branch, identify the single support interval and square-root edges, and thereby determine the almost-sure limits of the empirical extreme eigenvalues.
\item \textit{One-sample Marchenko--Pastur and LSS consequences.}  For a compactly supported positive population spectrum, deformed Marchenko--Pastur convergence gives the normalized first-order spectral limit for every $c\in(0,\infty)$ throughout the i.i.d. finite-fourth-moment class.  Under the exact identity null the limit is strictly positive.  The real-Gaussian Bai--Silverstein fluctuation result remains in the analytic nonsingular regime $c<1$.
\item \textit{Penalty scaling and statistical consequences.}  Besides the deterministic Gaussian phase diagram, a joint random-matrix/penalty limit shows that sample-covariance noise alone generates an $O(p)$ discrepancy and hence the same critical scale $\tau_p\asymp p$.  At the deterministic critical scale, balanced-Wasserstein consistency still passes through a dimension-free Lipschitz transform.  Secondary testing diagnostics remain in the Supplementary Numerical Material.
\end{enumerate}

The Marchenko--Pastur and Bai--Silverstein theorems, strong/Haar freeness, multiplicative subordination, and free-convolution regularity serve as external inputs.  Against this background, the model-specific contributions are the KL-UOT ridge-product reduction and positive-semidefinite continuation, the explicit symmetric $\eta$-transform and algebraic physical branch with edge identification, and the joint random-matrix/penalty phase diagram.  Proposition~\ref{prop:free-ridge-transfer} provides the functional-calculus bridge between these ingredients.

The asymptotic regimes differ across the results.  The first-order one-sample theory allows every $c>0$, compactly supported positive population spectra, and i.i.d. entries with finite fourth moment; the real-Gaussian LSS calibration remains in the nonsingular regime $c<1$.  The identity-Wishart two-sample weak and strong limits allow $c_i>0$, whereas the explicit algebraic support and edge regularity are developed for $c_i<1$.  Independent Haar orientations give deformed-population first-order limits.  When finite sample covariances are singular, we use only the positive-semidefinite spectral extension and make no singular-Gaussian KL interpretation.

Sections~\ref{sec:spectral}--\ref{sec:critical-estimation} develop the spectral representation, one-sample benchmark, and two-sample ridge-product theory.  Section~\ref{sec:phase} treats penalty scaling, followed by numerical illustrations in Section~\ref{sec:numerics}.  Mathematical derivations deferred from the main line are collected in Appendices~\ref{app:gaussian-derivations}--\ref{app:one-sample-consequences}, while computational diagnostics remain in the Supplementary Numerical Material.

\section{Spectral formulation of Gaussian KL-UOT}\label{sec:spectral}

Throughout, for finite nonnegative measures $\rho$ and $\eta$ we use the generalized Kullback--Leibler divergence
\begin{equation}\label{eq:finite-kl}
\KL(\rho\mid\eta)=
\begin{cases}
\displaystyle\int(r\log r-r+1)\,\dd\eta,&\rho=r\eta,\\[1ex]
+\infty,&\rho\not\ll\eta,
\end{cases}
\end{equation}
with the convention $0\log0=0$.

The finite-dimensional inputs needed below are the mass separation, Gaussian projection, mean resolvent, and covariance Riccati equation from~\cite{YangZhang2026}.  We collect them in one proposition, with derivations deferred to Appendix~\ref{app:gaussian-derivations}, so that the random-matrix development remains separate from the endpoint calculation.

\subsection{Finite-dimensional Gaussian identities}

For probability measures $\nu_0,\nu_1$, define
\begin{equation}\label{eq:normalized-action}
\Acal_{\tau_0,\tau_1}(\nu_0,\nu_1;\mu_0,\mu_1)
:=W_2^2(\nu_0,\nu_1)+\tau_0\KL(\nu_0\mid\mu_0)+\tau_1\KL(\nu_1\mid\mu_1),
\end{equation}
and let
\[
\Acal_*:=\inf_{\nu_0,\nu_1\in\mathcal P_2(\R^p)}
\Acal_{\tau_0,\tau_1}(\nu_0,\nu_1;\mu_0,\mu_1).
\]
Write $T=\tau_0+\tau_1$ and $w_i=\tau_i/T$.  For Gaussian references, set
\[
\mu_i=\N(m_i,\Sigma_i),\qquad \Sigma_i\in\mathbb S_{++}^p,
\]
\begin{equation}\label{eq:r-def}
r_i:=\frac{2}{\tau_i},\qquad \delta:=m_0-m_1,
\end{equation}
and
\begin{equation}\label{eq:C-def}
A_i=\Sigma_i^{-1},\qquad C_i=A_i+r_iI,
\qquad \kappa=r_1-r_0.
\end{equation}
Define further
\begin{equation}\label{eq:B-S-def}
B=C_1^{1/2}C_0C_1^{1/2},\qquad
s(t)=\frac{\kappa+\sqrt{\kappa^2+4t}}{2},\qquad
\mathsf S=s(B).
\end{equation}

\begin{proposition}\label{prop:gaussian-ingredients}
The following facts hold.
\begin{enumerate}[label=\textup{(\roman*)},leftmargin=2.2em]
\item \emph{Mass separation.}  If $\alpha=a\mu_0$ and $\beta=b\mu_1$ with $a,b>0$, then
\begin{equation}\label{eq:mass-separation}
\Ucal_{\tau_0,\tau_1}(\alpha,\beta)
=\tau_0a+\tau_1b-TM_*,
\end{equation}
where
\begin{equation}\label{eq:optimal-mass}
M_*=a^{w_0}b^{w_1}\exp\!\left(-\frac{\Acal_*}{T}\right).
\end{equation}

\item \emph{Gaussian reduction.}  If $\nu_0,\nu_1\in\mathcal P_2(\R^p)$ have means $u,v$ and positive-definite covariance matrices $P,Q$, and $g_0=\N(u,P)$, $g_1=\N(v,Q)$, then
\begin{align}
W_2^2(\nu_0,\nu_1)&\ge W_2^2(g_0,g_1),\label{eq:gelbrich-reduction}\\
\KL(\nu_0\mid\mu_0)&\ge \KL(g_0\mid\mu_0),\qquad
\KL(\nu_1\mid\mu_1)\ge \KL(g_1\mid\mu_1).\label{eq:kl-gaussian-projection}
\end{align}
Restricting the two shape measures to nondegenerate Gaussians therefore leaves the infimum defining $\Acal_*$ unchanged, and every finite-objective minimizer is Gaussian.  The Gaussian covariance minimum is attained at an interior pair $(P_*,Q_*)\in\mathbb S_{++}^p\times\mathbb S_{++}^p$.

\item \emph{Mean action.}  The optimal mean contribution is
\begin{equation}\label{eq:mean-action}
\Acal_{\rm mean}
=\delta^\top\bigl(I+r_0\Sigma_0+r_1\Sigma_1\bigr)^{-1}\delta.
\end{equation}
The adjusted means are
\begin{equation}\label{eq:adjusted-means}
u_*=m_0-r_0\Sigma_0h_*,\qquad
v_*=m_1+r_1\Sigma_1h_*,
\end{equation}
where
\[
h_*=(I+r_0\Sigma_0+r_1\Sigma_1)^{-1}\delta.
\]

\item \emph{Covariance stationarity and Riccati reduction.}  If $L_*\succ0$ is the optimal Gaussian transport map, so that $Q_*=L_*P_*L_*$, then
\begin{equation}\label{eq:cov-foc}
P_*^{-1}=A_0+r_0(I-L_*),\qquad
Q_*^{-1}=A_1+r_1(I-L_*^{-1}),
\end{equation}
and $L_*$ is the unique positive-definite solution of
\begin{equation}\label{eq:riccati}
L_*C_1L_*-\kappa L_*=C_0.
\end{equation}
Equivalently, the optimal transport map is
\begin{equation}\label{eq:L-star}
L_*=C_1^{-1/2}\mathsf S C_1^{-1/2}.
\end{equation}
\end{enumerate}
\end{proposition}

These finite-dimensional identities provide the input for the spectral reparameterization used throughout the random-matrix analysis.

\subsection{Spectral representation}

\begin{theorem}[Exact spectral representation]\label{thm:exact-spectral}
For nondegenerate Gaussian inputs, the optimal action decomposes as
\[
\Acal_*=\Acal_{\rm mean}+\Acal_{\rm cov},
\]
where \eqref{eq:mean-action} holds and the covariance contribution is
\begin{equation}\label{eq:general-cov-spectral}
\begin{aligned}
\Acal_{\rm cov}
={}&\frac{\tau_0}{2}\log\det\Sigma_0
+\frac{\tau_1}{2}\log\det\Sigma_1
+\frac{\tau_1-\tau_0}{2}\log\det C_1 \\
&+\frac{\tau_0-\tau_1}{2}\log\det \mathsf S
+\frac{T}{2}\log\det(\mathsf S-r_1I).
\end{aligned}
\end{equation}
Equivalently, if $\lambda_1(B),\ldots,\lambda_p(B)$ are the eigenvalues of $B$, then
\begin{equation}\label{eq:LSS-general}
\Acal_{\rm cov}
=\frac{\tau_0}{2}\log\det\Sigma_0
+\frac{\tau_1}{2}\log\det\Sigma_1
+\frac{\tau_1-\tau_0}{2}\log\det C_1
+\sum_{j=1}^p\Phi_{\tau_0,\tau_1}(\lambda_j(B)),
\end{equation}
where
\begin{equation}\label{eq:Phi}
\Phi_{\tau_0,\tau_1}(t)
=\frac{\tau_0-\tau_1}{2}\log s(t)
+\frac{T}{2}\log(s(t)-r_1).
\end{equation}
In addition, $\mathsf S-r_1I\succ0$.
\end{theorem}

\begin{proof}
Let $P_*,Q_*$ be the adjusted covariance matrices.  Since the optimal coupling between $\N(0,P_*)$ and $\N(0,Q_*)$ is induced by $L_*$, one has $Q_*=L_*P_*L_*$.  By part~(iv) of Proposition~\ref{prop:gaussian-ingredients},
\[
P_*^{-1}=A_0+r_0(I-L_*),\qquad
Q_*^{-1}=A_1+r_1(I-L_*^{-1}).
\]
The covariance contribution to the Wasserstein term is
\[
\tr P_*+\tr Q_*-2\tr(L_*P_*).
\]
Multiplying the first identity in \eqref{eq:cov-foc} by $P_*$ and the second by $Q_*$, and using $Q_*=L_*P_*L_*$, gives
\[
\frac{\tau_0}{2}\{\tr(A_0P_*)-p\}
=-\tr(P_*-L_*P_*),
\]
and
\[
\frac{\tau_1}{2}\{\tr(A_1Q_*)-p\}
=-\tr(Q_*-L_*P_*).
\]
The two trace contributions cancel the Bures trace term, leaving
\begin{equation}\label{eq:log-only}
\Acal_{\rm cov}
=\frac{\tau_0}{2}\log\frac{\det\Sigma_0}{\det P_*}
+\frac{\tau_1}{2}\log\frac{\det\Sigma_1}{\det Q_*}.
\end{equation}

By \eqref{eq:riccati},
\[
P_*^{-1}=C_0-r_0L_*=L_*(C_1L_*-r_1I),
\]
where we used $\kappa+r_0=r_1$.  Since $C_1L_*$ is similar to $\mathsf S$ and
\[
\det L_*=\frac{\det \mathsf S}{\det C_1},
\]
we obtain
\begin{equation}\label{eq:det-P}
\det P_*^{-1}=\frac{\det \mathsf S\,\det(\mathsf S-r_1I)}{\det C_1}.
\end{equation}
Similarly, the second stationarity identity gives
\[
Q_*^{-1}=C_1-r_1L_*^{-1}=(C_1L_*-r_1I)L_*^{-1},
\]
so
\begin{equation}\label{eq:det-Q}
\det Q_*^{-1}=\frac{\det C_1\,\det(\mathsf S-r_1I)}{\det \mathsf S}.
\end{equation}
Substituting \eqref{eq:det-P}--\eqref{eq:det-Q} into \eqref{eq:log-only} yields \eqref{eq:general-cov-spectral}, and functional calculus yields \eqref{eq:LSS-general}.

To verify that the logarithms are well defined, it remains to show $\mathsf S-r_1I\succ0$.  Since $C_0\succ r_0I$ and $C_1\succ r_1I$,
\[
B=C_1^{1/2}C_0C_1^{1/2}\succ r_0C_1\succ r_0r_1I.
\]
The scalar map $s$ is increasing and $s(r_0r_1)=r_1$, because $r_1(r_1-\kappa)=r_0r_1$.  Since $B\succ r_0r_1I$, functional calculus gives $\mathsf S=s(B)\succ r_1I$.
\end{proof}

\subsection{Equal-penalty ridge products}

Although the preceding formula is spectral, its asymmetric form obscures the matrix structure that survives after the substitution $\Sigma_i\mapsto S_i$.  With equal penalties, the determinant terms combine into a nonlinear product of ridge transforms, which is the form needed for the high-dimensional analysis.  Most of what follows uses equal marginal penalties:
\begin{equation}\label{eq:equal-tau}
\tau_0=\tau_1=\tau,\qquad r=2/\tau.
\end{equation}
Define the ridge contractions
\begin{equation}\label{eq:R-def}
R_i=r\Sigma_i(I+r\Sigma_i)^{-1}=r(\Sigma_i^{-1}+rI)^{-1},
\end{equation}
and
\begin{equation}\label{eq:Rprod}
\mathcal R=R_1^{1/2}R_0R_1^{1/2}.
\end{equation}
Then $0\prec R_i\prec I$ and $0\prec\mathcal R\prec I$.  Each eigenvalue of $R_i$ has the form $r\lambda/(1+r\lambda)\in(0,1)$, and
\[
\|\mathcal R\|_{\rm op}
\le \|R_1\|_{\rm op}\|R_0\|_{\rm op}<1.
\]
All logarithms and principal square roots in the ridge-product representation are therefore evaluated away from their branch singularities.

\begin{lemma}[Ridge-product symmetry]\label{lem:ridge-product-symmetry}
Let $A,B\in\mathbb S_+^p$.  Then
\[
A^{1/2}BA^{1/2}
\qquad\text{and}\qquad
B^{1/2}AB^{1/2}
\]
have the same characteristic polynomial and hence the same spectrum, with algebraic multiplicities.  If $A,B\in\mathbb S_{++}^p$, the two matrices are similar.  Consequently, exchanging $R_0$ and $R_1$ in the ridge product leaves every spectral statistic of the product unchanged; in particular,
\begin{equation}\label{eq:ridge-product-exchange}
\det\!\left(I-\{R_1^{1/2}R_0R_1^{1/2}\}^{1/2}\right)
=
\det\!\left(I-\{R_0^{1/2}R_1R_0^{1/2}\}^{1/2}\right).
\end{equation}
\end{lemma}

\begin{proof}
Set $X=A^{1/2}B^{1/2}$ and $Y=B^{1/2}A^{1/2}$.  Then $XY=A^{1/2}BA^{1/2}$ and $YX=B^{1/2}AB^{1/2}$.  Sylvester's determinant identity gives
\[
\det(\lambda I-XY)=\det(\lambda I-YX)
\]
for $\lambda\ne0$, and hence for every $\lambda$ by polynomial identity.  Thus the characteristic polynomials coincide.  If $A$ and $B$ are positive definite, $X$ is invertible and
\[
YX=X^{-1}(XY)X,
\]
so the two matrices are similar.  Applying the spectral conclusion to $A=R_1$ and $B=R_0$, and then the continuous functional calculus for the principal square root, yields \eqref{eq:ridge-product-exchange}.
\end{proof}

\begin{corollary}[Ridge-product identity]\label{cor:ridge-product}
Under \eqref{eq:equal-tau}, the covariance action admits the ridge-product representation
\begin{equation}\label{eq:ridge-formula}
\Acal_{\rm cov,\tau}(\Sigma_0,\Sigma_1)
=-\frac{\tau}{2}\log\det(I-R_0)
-\frac{\tau}{2}\log\det(I-R_1)
+\tau\log\det(I-\mathcal R^{1/2}).
\end{equation}
In particular, the associated determinants satisfy
\begin{equation}\label{eq:det-ineq}
\det(I-\mathcal R^{1/2})^2
\ge \det(I-R_0)\det(I-R_1),
\end{equation}
with equality if and only if $\Sigma_0=\Sigma_1$.
\end{corollary}

\begin{proof}
When $\kappa=0$, $\mathsf S=(C_1^{1/2}C_0C_1^{1/2})^{1/2}$ and \eqref{eq:general-cov-spectral} becomes
\[
\Acal_{\rm cov,\tau}
=\frac{\tau}{2}(\log\det\Sigma_0+\log\det\Sigma_1)
+\tau\log\det(\mathsf S-rI).
\]
Factor $\mathsf S-rI=\mathsf S(I-r\mathsf S^{-1})$.  Because $\det\mathsf S=(\det C_0\det C_1)^{1/2}$,
\[
\frac{\tau}{2}\log\det(\Sigma_0C_0)
+\frac{\tau}{2}\log\det(\Sigma_1C_1)
=\frac{\tau}{2}\sum_{i=0}^1\log\det(I+r\Sigma_i)
=-\frac{\tau}{2}\sum_{i=0}^1\log\det(I-R_i).
\]
Also,
\[
(r\mathsf S^{-1})^2
=r^2C_1^{-1/2}C_0^{-1}C_1^{-1/2}
=R_1^{1/2}R_0R_1^{1/2}=\mathcal R,
\]
so $r\mathsf S^{-1}=\mathcal R^{1/2}$, which gives \eqref{eq:ridge-formula}.  Since $\Acal_{\rm cov,\tau}\ge0$, exponentiation yields \eqref{eq:det-ineq}.  Equality of the covariance action forces the adjusted marginals to coincide with their references and the Wasserstein term to vanish; hence $\Sigma_0=\Sigma_1$.  The converse is immediate.
\end{proof}

\begin{remark}\label{rem:ridge-symmetry}
The asymmetric appearance of $C_0,C_1$, and $\mathsf S$ in Theorem~\ref{thm:exact-spectral} is only a feature of the parametrization.  The endpoint objective itself is invariant under exchanging $(\tau_0,\Sigma_0)$ with $(\tau_1,\Sigma_1)$.  In the equal-penalty representation this symmetry is encoded by Lemma~\ref{lem:ridge-product-symmetry}: the two exchanged ridge products have the same full spectrum (and, here, are similar because $R_0,R_1\succ0$).  For $p=1$, direct scalar minimization of the two adjusted variances reproduces the specialization of Theorem~\ref{thm:exact-spectral}; this provides a useful sanity check on the logarithmic signs and penalty normalization.
\end{remark}

\section{One-sample linear spectral statistics}\label{sec:one-sample}

Let $X_p=(x_{kj})$ be an $n\times p$ matrix with i.i.d. standardized entries and let $\Sigma_p\succ0$ be deterministic.  We consider
\begin{equation}\label{eq:general-sample-cov}
S_p=\Sigma_p^{1/2}\frac{X_p^\top X_p}{n}\Sigma_p^{1/2},
\qquad \frac pn\to c\in(0,\infty).
\end{equation}
Throughout this section, $\tau>0$ is fixed.  Standard normal entries recover the Gaussian sampling model; the first-order spectral limit below requires only the stated sample-covariance assumptions.  Write $H_p=F^{\Sigma_p}$ for the empirical population spectral distribution.

\begin{remark}[Spectral versus Gaussian regimes]\label{rem:c-less-one}
The first-order spectral theory below is formulated for all $c\in(0,\infty)$.  The nondegenerate Gaussian interpretation is narrower: when $c<1$, the Bai--Yin lower edge is positive and $S_p\succ0$ eventually almost surely, so the spectral statistic agrees with the finite-dimensional Gaussian KL-UOT action.  At $c=1$ the limiting support touches the origin, and for $c>1$ the sample covariance has a null space of asymptotic proportion $1-1/c$.  In those regimes we retain the positive-semidefinite spectral extension but do not identify it with KL divergence between degenerate Gaussian measures.  This separates the random-matrix statement from the additional support-subspace analysis required for singular Gaussian KL-UOT.
\end{remark}

\subsection{Scalar spectral function}

\begin{proposition}[Scalar spectral function]\label{prop:f-tau}
For $x\ge0$, define
\begin{equation}\label{eq:f-tau}
f_\tau(x)
=\tau\log\left(\sqrt{(1+r)(1+rx)}-r\sqrt{x}\right),
\qquad r=2/\tau.
\end{equation}
At the endpoint, the continuous extension is finite and satisfies
\begin{equation}\label{eq:f-zero}
f_\tau(0)=\frac{\tau}{2}\log(1+r).
\end{equation}
For every $\Sigma\succ0$, the covariance action is the linear spectral statistic
\begin{equation}\label{eq:identity-action-lfs}
\Acal_{\rm cov,\tau}(I,\Sigma)
=\sum_{j=1}^p f_\tau(\lambda_j(\Sigma)).
\end{equation}
It also satisfies
\begin{equation}\label{eq:f-properties}
f_\tau(x)\ge0,\quad f_\tau(x)=0\iff x=1,
\quad f_\tau(1)=f_\tau'(1)=0,
\quad f_\tau''(1)=\frac{1}{2(1+2/\tau)}.
\end{equation}
A Taylor expansion at $x=1$ gives
\begin{equation}\label{eq:f-local}
f_\tau(1+h)=\frac{h^2}{4(1+2/\tau)}+O(h^3).
\end{equation}
In the large-penalty limit, $f_\tau$ converges pointwise:
\begin{equation}\label{eq:f-balanced-limit}
f_\tau(x)\to(\sqrt{x}-1)^2.
\end{equation}
\end{proposition}

\begin{proof}
When $\Sigma_0=I$, $R_0=r(1+r)^{-1}I$ and an eigenvalue $x>0$ of $\Sigma$ produces the ridge eigenvalue $rx/(1+rx)$.  Substitution into \eqref{eq:ridge-formula} gives \eqref{eq:f-tau} and \eqref{eq:identity-action-lfs}.  The same scalar formula has the finite continuous limit \eqref{eq:f-zero} as $x\downarrow0$, which defines $f_\tau$ on $[0,\infty)$.

To prove nonnegativity, set $t=\sqrt x$.  Since
\[
(1+r)(1+rt^2)-(1+rt)^2=r(t-1)^2\ge0,
\]
we have
\[
\sqrt{(1+r)(1+rx)}-r\sqrt x\ge1,
\]
with equality only at $x=1$.  This proves the nonnegativity and equality statement in \eqref{eq:f-properties}.  Direct differentiation yields $f_\tau'(1)=0$ and
\[
f_\tau''(1)=\frac{r\tau}{4(1+r)}=\frac{1}{2(1+r)}=\frac{1}{2(1+2/\tau)},
\]
which proves \eqref{eq:f-local}.  Expanding \eqref{eq:f-tau} at $r=0$ gives \eqref{eq:f-balanced-limit}.
\end{proof}

For a positive-semidefinite matrix $S\in\mathbb S_+^p$, define the spectral extension
\begin{equation}\label{eq:psd-spectral-extension}
\mathfrak A_\tau(S):=\sum_{j=1}^p f_\tau(\lambda_j(S)).
\end{equation}
When $S\succ0$, Proposition~\ref{prop:f-tau} gives
$\mathfrak A_\tau(S)=\Acal_{\rm cov,\tau}(I,S)$.  Definition~\eqref{eq:psd-spectral-extension} is used below whenever a discrete sampling law can produce a singular sample covariance at finite $p$; it is only a spectral extension, not a definition of KL divergence between degenerate Gaussian measures.  Under the assumptions of Theorem~\ref{thm:general-pop-limit}, the Bai--Yin lower-edge limit implies that $S_p\succ0$ eventually almost surely, so the extended statistic eventually coincides pathwise with the nondegenerate Gaussian KL-UOT covariance action.

\subsection{General-population spectral limit}

For each $p$, let $F_{c_p,H_p}$ denote the standard deterministic deformed Marchenko--Pastur law associated with $c_p=p/n$ and the empirical population law $H_p=F^{\Sigma_p}$.  Its companion Stieltjes transform $\underline m_{c_p,H_p}$ is characterized by the Silverstein equation below with $(c,H)$ replaced by $(c_p,H_p)$.  Its limit is denoted by $F_{c,H}$; the companion Stieltjes transform $\underline m_{c,H}$ is characterized on $\mathbb C^+$ by
\begin{equation}\label{eq:silverstein-equation}
z=-\frac{1}{\underline m_{c,H}(z)}
+c\int\frac{t}{1+t\underline m_{c,H}(z)}\,\dd H(t),
\qquad \Im\underline m_{c,H}(z)>0.
\end{equation}
The Stieltjes transform of $F_{c,H}$ is related to the companion transform by
\begin{equation}\label{eq:companion-relation-general}
\underline m_{c,H}(z)=-\frac{1-c}{z}+c\,m_{c,H}(z).
\end{equation}

\begin{theorem}[Deformed Marchenko--Pastur limit]\label{thm:general-pop-limit}
Fix $\tau>0$.
Assume that $c_p=p/n\to c\in(0,\infty)$, that $H_p\Rightarrow H$ with
\[
\operatorname{supp}(H_p)\subset[\underline\lambda,\overline\lambda]
\subset(0,\infty)
\]
uniformly in $p$, and that, for all $p$, the entries of $X_p$ are independent copies of a fixed real random variable $x$ satisfying $\mathbb E x=0$, $\mathbb E x^2=1$, and $\mathbb E|x|^4<\infty$.  Define
\begin{equation}\label{eq:theta-finite-p}
\Theta_{\tau,p}:=\int f_\tau(x)\,\dd F_{c_p,H_p}(x),
\qquad
\Theta_\tau(c,H):=\int f_\tau(x)\,\dd F_{c,H}(x).
\end{equation}
Then the empirical statistic and its deterministic counterpart satisfy, almost surely,
\begin{equation}\label{eq:general-pop-de}
\frac1p\mathfrak A_\tau(S_p)-\Theta_{\tau,p}\longrightarrow0,
\end{equation}
and
\begin{equation}\label{eq:general-pop-limit}
\Theta_{\tau,p}\longrightarrow\Theta_\tau(c,H),
\qquad
\frac1p\mathfrak A_\tau(S_p)\longrightarrow\Theta_\tau(c,H).
\end{equation}
\end{theorem}

\begin{remark}\label{rem:finite-p-reference-scope}
Equation~\eqref{eq:general-pop-de} is a first-order comparison: both the empirical law $F^{S_p}$ and the deterministic deformed Marchenko--Pastur law $F_{c_p,H_p}$ converge to $F_{c,H}$.  No finite-size rate such as $O(p^{-1})$ or resolvent-level deterministic equivalent is asserted, nor do we claim that $\Theta_{\tau,p}$ is uniformly more accurate than $\Theta_\tau(c,H)$ at finite $p$.  Quantitative bounds of that type would require additional control of the resolvent bias.
\end{remark}

\begin{proof}
The general Marchenko--Pastur theorem gives $F^{S_p}\Rightarrow F_{c,H}$ almost surely, while continuity of the deformed Marchenko--Pastur map under $c_p\to c$ and $H_p\Rightarrow H$ gives $F_{c_p,H_p}\Rightarrow F_{c,H}$~\cite{SilversteinChoi1995,BaiSilverstein2010}.  For the random matrices, the Bai--Yin upper-edge theorem~\cite{BaiYin1993}, whose conclusion requires only the finite fourth moment used here, gives
\[
\lambda_{\max}(X_p^\top X_p/n)\to(1+\sqrt c)^2
\qquad\text{a.s.}
\]
Hence, after choosing any $c^*>c$ with $c_p\le c^*$ eventually,
\[
\operatorname{spec}(S_p)
\subset
\bigl[0,\overline\lambda(1+\sqrt{c^*})^2+o(1)\bigr]
\]
almost surely.  If $c<1$, the Bai--Yin lower edge additionally yields eventual positive definiteness, but no lower spectral gap is needed for the spectral limit because $f_\tau$ is continuous at zero.  The deterministic deformed Marchenko--Pastur law has the same uniform upper envelope.  Indeed, using $F_{c_p,H_p}=H_p\boxtimes\mu_{\mathrm{MP},c_p}$~\cite{NicaSpeicher2006,BaiSilverstein2010}, let free positive operators $A,B$ satisfy
\[
\operatorname{law}(A)=H_p,\qquad
\operatorname{law}(B)=\mu_{\mathrm{MP},c_p}.
\]
Then $0\preceq B^{1/2}AB^{1/2}\preceq\overline\lambda B$, so
\begin{equation}\label{eq:deterministic-mp-envelope}
\operatorname{supp}F_{c_p,H_p}
\subset
\bigl[0,\overline\lambda(1+\sqrt{c^*})^2\bigr]
\end{equation}
for all sufficiently large $p$.  Since $f_\tau$ is continuous on this common compact envelope,
\[
\int f_\tau\,\dd F^{S_p}\to\int f_\tau\,\dd F_{c,H},
\qquad
\Theta_{\tau,p}\to\int f_\tau\,\dd F_{c,H}.
\]
Definition~\eqref{eq:psd-spectral-extension} then proves both \eqref{eq:general-pop-de} and \eqref{eq:general-pop-limit}.  When $c<1$, the Bai--Yin lower-edge limit also gives $S_p\succ0$ for all sufficiently large $p$ almost surely, so the same limits then apply eventually to the genuine nondegenerate-Gaussian covariance action.  For $c\ge1$, the statement remains a positive-semidefinite spectral limit.
\end{proof}

\begin{remark}\label{rem:first-order-universality}
The underlying transport functional is Gaussian, but the first-order spectral statistic need not be sampled from a Gaussian law: \eqref{eq:general-pop-de} holds throughout the stated finite-moment class.  This robustness concerns the global first-order limit, not bulk or edge universality.  At second order, non-Gaussian samples introduce the usual cumulant and eigenvector corrections~\cite{NajimYao2016}, so the fluctuation result below is stated for real Gaussian data.
\end{remark}

\subsection{Identity-population bias}

Let $\mu_{\mathrm{MP},c}$ denote the Marchenko--Pastur law with aspect ratio $c\in(0,\infty)$,
\begin{equation}\label{eq:mp-density}
\begin{aligned}
\dd\mu_{\mathrm{MP},c}(x)
={}&\left(1-\frac1c\right)_+\delta_0(\dd x)\\
&+\frac{\sqrt{(b_c-x)(x-a_c)}}{2\pi cx}
\mathbf 1_{[a_c,b_c]}(x)\,\dd x,
\end{aligned}
\qquad
a_c=(1-\sqrt c)^2,\quad b_c=(1+\sqrt c)^2.
\end{equation}
For $c>1$ the continuous part has total mass $1/c$ and the remaining mass is the deterministic rank-deficiency atom at the origin.
Define
\begin{equation}\label{eq:bias-b}
b_{\tau,c}:=\int f_\tau(x)\,\dd\mu_{\mathrm{MP},c}(x).
\end{equation}

\begin{corollary}[Identity-population bias]\label{cor:mp-bias}
Under the assumptions of \cref{thm:general-pop-limit} with $H=\delta_1$, the normalized spectral statistic converges almost surely according to
\begin{equation}\label{eq:mp-limit}
\frac1p\mathfrak A_\tau(S_p)
\xrightarrow{\as}b_{\tau,c}>0.
\end{equation}
Moreover, as $c\downarrow0$,
\begin{equation}\label{eq:small-c-bias}
b_{\tau,c}
=\frac{c}{4(1+2/\tau)}+O(c^2).
\end{equation}
\end{corollary}

\begin{proof}
Since $F_{c,\delta_1}=\mu_{\mathrm{MP},c}$, \cref{thm:general-pop-limit} applies directly.  The spectrum is eventually contained in a fixed compact subset of $[0,\infty)$, so continuity of $f_\tau$ and Definition~\eqref{eq:psd-spectral-extension} give \eqref{eq:mp-limit}.  Strict positivity follows because $f_\tau$ vanishes only at $1$, whereas $\mu_{\mathrm{MP},c}$ is nondegenerate.

For the small-$c$ expansion, $f_\tau$ is $C^4$ on a fixed neighborhood of $1$, and $f_\tau(1)=f_\tau'(1)=0$.  Its Taylor expansion is
\[
f_\tau(1+h)
=\frac{h^2}{4(1+2/\tau)}+\frac{f_\tau^{(3)}(1)}{6}h^3+O(h^4)
\]
uniformly for $|h|$ in that neighborhood.  For all sufficiently small $c$, the Marchenko--Pastur support is contained in this neighborhood, and its centered moments satisfy
\[
\int(x-1)^2\,\dd\mu_{\mathrm{MP},c}(x)=c,\qquad
\int(x-1)^3\,\dd\mu_{\mathrm{MP},c}(x)=c^2,
\]
\[
\int(x-1)^4\,\dd\mu_{\mathrm{MP},c}(x)=2c^2+c^3.
\]
Integrating the fourth-order Taylor expansion gives
\[
b_{\tau,c}=\frac{c}{4(1+2/\tau)}+O(c^2).
\]
\end{proof}

If the finite-$p$ null is exact, $\Sigma_p=I_p$ for every $p$, then the population action is $\Acal_{\rm cov,\tau}(I,I)=0$ while the spectral statistic has the strictly positive limit in \eqref{eq:mp-limit}.  For $c<1$, eventual positive definiteness makes $\mathfrak A_\tau(S_p)$ the genuine nondegenerate-Gaussian plug-in action, so the usual plug-in estimator is inconsistent under the exact null.  For $c\ge1$, only the positive-semidefinite spectral extension is asserted.  The weaker assumption $H_p\Rightarrow\delta_1$ does not by itself impose the finite-$p$ identity null.

\begin{remark}\label{rem:null-mass-collapse}
Consider equal unit masses, a common fixed penalty $\tau$, a common zero mean, and an identity-population sample covariance $S_p$.  Define the spectral proxies
\[
\widehat M_{*,p}^{\rm spec}:=\exp\left\{-\frac{\mathfrak A_\tau(S_p)}{2\tau}\right\},
\qquad
\widehat\Ucal_{p,\tau}^{\rm spec}:=2\tau(1-\widehat M_{*,p}^{\rm spec}).
\]
If $\Theta_\tau(c,H)>0$, Theorem~\ref{thm:general-pop-limit} immediately yields, almost surely,
\begin{equation}\label{eq:mass-collapse-general}
-\frac1p\log\widehat M_{*,p}^{\rm spec}\longrightarrow\frac{\Theta_\tau(c,H)}{2\tau},
\qquad
\widehat M_{*,p}^{\rm spec}\longrightarrow0,
\qquad
\widehat\Ucal_{p,\tau}^{\rm spec}\longrightarrow2\tau.
\end{equation}
For the identity population, $H=\delta_1$ and $\Theta_\tau(c,\delta_1)=b_{\tau,c}>0$.  When $c<1$, eventual positive definiteness makes these spectral proxies coincide from some index onward with the genuine Gaussian plug-in mass and UOT value.  For every $c>0$, sample-covariance noise therefore produces an exponentially small spectral transported-mass proxy at fixed penalty; in the nonsingular regime this is exactly the Gaussian plug-in mass collapse, while for $c\ge1$ it records the rank-deficient spectral continuation.
\end{remark}

\subsection{Gaussian fluctuations}

For an analytic test function $f$, let $\mathfrak m_c^{\rm BS}(f)$ and $\mathfrak v_c^{\rm BS}(f)$ denote the centering and variance functionals in the real Gaussian Bai--Silverstein LSS CLT~\cite{BaiSilverstein2004}.  Appendix~\ref{app:BS-calibration} records their companion-Stieltjes-transform representations and verifies the analyticity of $f_\tau$; related high-dimensional LSS calibrations appear in~\cite{ChenZhengZou2026}.

\begin{corollary}[Bai--Silverstein fluctuations]\label{cor:lss-clt}
Fix $\tau>0$, assume $\Sigma_p=I_p$, let the entries of $X_p$ be independent $N(0,1)$ random variables, and let $c_p=p/n\to c\in(0,1)$.  Put
\[
b_{\tau,c_p}=\int f_\tau(x)\,\dd\mu_{\mathrm{MP},c_p}(x).
\]
Then the centered spectral statistic satisfies the central limit theorem
\begin{equation}\label{eq:uot-clt}
\Acal_{\rm cov,\tau}(I,S_p)-p b_{\tau,c_p}
\Longrightarrow
\N\bigl(\mathfrak m_c^{\rm BS}(f_\tau),\mathfrak v_c^{\rm BS}(f_\tau)\bigr).
\end{equation}
The variance is strictly positive.
\end{corollary}

\begin{proof}
Because $c<1$, the interval $[a_c,b_c]$ is compactly contained in $(0,\infty)$.  Choose the principal square-root and logarithm branches on a simply connected complex neighborhood of this interval that avoids the nonpositive real axis and the zeros of the analytic continuation of the argument in \eqref{eq:f-tau}.  The positivity argument in \cref{prop:f-tau} and compactness of the limiting support allow this neighborhood to be chosen after shrinking it if necessary.  On this neighborhood, $f_\tau$ is analytic, and \eqref{eq:identity-action-lfs} is precisely the corresponding LSS.  The real Gaussian Bai--Silverstein theorem~\cite[Theorem~1.1]{BaiSilverstein2004} then yields \eqref{eq:uot-clt}.  Strict positivity of the variance follows from the equivalent Joukowski/Fourier representation in Appendix~\ref{app:BS-calibration}, since $f_\tau$ is nonconstant on the Marchenko--Pastur support.
\end{proof}

\subsection{Covariance identity testing}\label{sec:test}

The LSS CLT also gives a test of covariance identity.  Consider $H_0:\Sigma=I_p$ based on i.i.d. observations $X_1,\ldots,X_n\sim\N(0,\Sigma)$, initially with known zero mean.  Define
\begin{equation}\label{eq:test-z}
T_{p,n}^{(\tau)}:=\sum_{j=1}^p f_\tau(\lambda_j(S_p)),\qquad
Z_{p,n}^{(\tau)}:=\frac{T_{p,n}^{(\tau)}-p b_{\tau,c_p}-\mathfrak m_{c_p}^{\rm BS}(f_\tau)}
{\sqrt{\mathfrak v_{c_p}^{\rm BS}(f_\tau)}}.
\end{equation}

By continuity in $c$ of the Bai--Silverstein centering and variance functionals for fixed analytic $f_\tau$, Corollary~\ref{cor:lss-clt} and Slutsky's theorem give
\[
Z_{p,n}^{(\tau)}\Longrightarrow\N(0,1)
\qquad\text{under }H_0.
\]
The one-sided rule
\[
\varphi_{p,n}^{(\tau)}:=\mathbf 1\{Z_{p,n}^{(\tau)}>z_{1-\alpha}\}
\]
has asymptotic level $\alpha$.

The separated-alternative consistency result and dense local population expansion are recorded in Appendix~\ref{app:one-sample-consequences}.  The test is not omnibus: spectrally invisible alternatives require additional information.  If the mean is unknown, the centered covariance with divisor $n-1$ has the same Gaussian null calibration after replacing $p/n$ by $p/(n-1)$~\cite{ZhengBaiYao2015}; testing the full Gaussian law would add a mean quadratic-form term.

This closes the scalar LSS benchmark.  We now turn to the genuinely noncommutative two-sample problem, where the KL-UOT action depends on the ridge-product spectrum and hence on relative eigenvectors.

\section{Two-sample ridge-product spectra}\label{sec:critical-estimation}

\subsection{Wishart products and free convolution}\label{sec:free}

We begin with the fixed-penalty regime, where the nonlinear ridge-product structure remains visible at leading order.  Dimension-dependent penalties are deferred to Section~\ref{sec:phase}.  Let $S_0,S_1$ be independent identity-Wishart sample covariances
\[
S_i=\frac1{n_i}X_i^\top X_i,
\qquad
X_i\in\R^{n_i\times p},
\qquad
\frac{p}{n_i}\to c_i\in(0,\infty),
\]
where the entries of $X_0,X_1$ are independent standard Gaussians.  The spectral theory below allows every $c_i>0$; when $c_i<1$ the sample covariances are nondegenerate almost surely for all sufficiently large $p$, whereas $c_i\ge1$ is interpreted through the positive-semidefinite spectral extension.  Put
\begin{equation}\label{eq:h-r}
h_r(x)=\frac{rx}{1+rx},\qquad r=2/\tau,
\end{equation}
and let
\begin{equation}\label{eq:rho-def}
\rho_{c,r}=(h_r)_\#\mu_{\mathrm{MP},c}
\end{equation}
be the push-forward of the Marchenko--Pastur law.

\begin{definition}[PSD spectral extension]\label{def:two-sample-psd-extension}
For $A,B\in\mathbb S_+^p$, set
\[
R_r(A):=h_r(A)=rA(I+rA)^{-1},
\qquad
\mathcal R_r(A,B):=R_r(B)^{1/2}R_r(A)R_r(B)^{1/2},
\]
and define
\begin{equation}\label{eq:two-sample-psd-extension}
\begin{aligned}
\mathfrak A_{\tau}^{(2)}(A,B)
:={}&-\frac{\tau}{2}\log\det(I-R_r(A))
     -\frac{\tau}{2}\log\det(I-R_r(B))\\
&+\tau\log\det\bigl(I-\mathcal R_r(A,B)^{1/2}\bigr).
\end{aligned}
\end{equation}
\end{definition}

The definition is well posed.  Indeed, with
\[
q_A:=\frac{r\|A\|_{\rm op}}{1+r\|A\|_{\rm op}},
\qquad
q_B:=\frac{r\|B\|_{\rm op}}{1+r\|B\|_{\rm op}},
\]
one has $q_A,q_B<1$ and
\[
0\preceq R_r(A)\preceq q_A I,
\qquad
0\preceq R_r(B)\preceq q_B I,
\qquad
0\preceq\mathcal R_r(A,B)\preceq q_Aq_B I.
\]
Hence
\[
I-R_r(A)\succ0,
\qquad
I-R_r(B)\succ0,
\qquad
I-\mathcal R_r(A,B)^{1/2}\succeq
(1-\sqrt{q_Aq_B})I\succ0,
\]
so the three logarithms in \eqref{eq:two-sample-psd-extension} are finite.  If $A,B\succ0$, Corollary~\ref{cor:ridge-product} yields $\mathfrak A_{\tau}^{(2)}(A,B)=\Acal_{\rm cov,\tau}(A,B)$.  For singular inputs, \eqref{eq:two-sample-psd-extension} is only a spectral continuation and is not a definition of KL divergence between degenerate Gaussian measures.

\begin{remark}\label{rem:psd-continuity}
The preceding definition is the continuous extension of the equal-penalty Gaussian covariance action from $\mathbb S_{++}^p\times\mathbb S_{++}^p$ to $\mathbb S_+^p\times\mathbb S_+^p$.  More precisely, regularization from the positive-definite cone yields, for every $A,B\succeq0$,
\begin{equation}\label{eq:psd-continuity}
\mathfrak A_{\tau}^{(2)}(A,B)
=\lim_{\varepsilon\downarrow0}
\Acal_{\rm cov,\tau}(A+\varepsilon I,B+\varepsilon I).
\end{equation}
Indeed, $A\mapsto h_r(A)$ is norm-continuous on the positive-semidefinite cone, the principal square root is norm-continuous on positive-semidefinite matrices, and the bounds in Definition~\ref{def:two-sample-psd-extension} leave the three logarithmic arguments uniformly positive for all sufficiently small $\varepsilon$.  Continuity of $\log\det$ on the positive-definite cone then gives \eqref{eq:psd-continuity}.  This statement concerns the spectral functional only and does not identify the right-hand limit with a KL divergence between singular Gaussian measures.
\end{remark}

\begin{lemma}\label{lem:functional-calculus-freeness}
Let $(M_{0,p},M_{1,p})$ be real symmetric matrix pairs whose spectra are eventually contained in fixed compact sets $K_0,K_1\subset\mathbb R$.  Suppose that, almost surely, these pairs converge in normalized mixed moments to a free self-adjoint pair $(x_0,x_1)$ in a tracial $C^*$-probability space $(\mathcal A,\varphi)$, with $\operatorname{spec}(x_i)\subset K_i$.  Thus, for every noncommutative polynomial $Q$,
\[
\frac1p\tr Q(M_{0,p},M_{1,p})
\longrightarrow \varphi\!\left(Q(x_0,x_1)\right).
\]
If $f\in C(K_0)$ and $g\in C(K_1)$, then
\[
(f(M_{0,p}),g(M_{1,p}))
\longrightarrow (f(x_0),g(x_1))
\]
in normalized mixed moments.  Continuous scalar functional calculus therefore preserves first-order asymptotic freeness under a common compact spectral envelope.
\end{lemma}

\begin{proof}
Choose polynomials $P_m,Q_m$ converging uniformly to $f,g$ on $K_0,K_1$.  For every fixed noncommutative polynomial $W$ in two variables, the assumed asymptotic freeness gives the normalized-trace limit of $W(P_m(M_{0,p}),Q_m(M_{1,p}))$ for each fixed $m$.  On the common spectral envelope, a telescoping expansion of $W$ and the bounds
\begin{align*}
\frac1p|\tr X|&\le\|X\|_{\rm op},\\
\|P_m(M_{0,p})-f(M_{0,p})\|_{\rm op}
&\le\|P_m-f\|_{\infty,K_0}.
\end{align*}
The analogous estimate holds for $g$.  A telescoping expansion then shows that replacing the polynomial approximants by the continuous functional-calculus matrices changes each normalized mixed trace by $o_m(1)$ uniformly in $p$.  Passing first $p\to\infty$ and then $m\to\infty$ proves the claim.  Since $f(x_0)$ and $g(x_1)$ belong to the unital algebras generated by the free variables $x_0$ and $x_1$, they are free.
\end{proof}

\begin{proposition}[Free ridge-product transfer]\label{prop:free-ridge-transfer}
Fix $\tau>0$ and $r=2/\tau$.  Let $M_{0,p},M_{1,p}\in\mathbb S_{+}^p$ be random symmetric positive-semidefinite matrices such that, for some finite constants $L_0,L_1$, their spectra are almost surely eventually contained in $[0,L_i]$, respectively.  Suppose that, almost surely, $(M_{0,p},M_{1,p})$ converges in normalized mixed moments to a free pair of positive operators $(x_0,x_1)$ with laws $\mu_i$ supported in $[0,L_i]$.  Set
\[
\rho_i=(h_r)_\#\mu_i,\qquad i=0,1.
\]
Then the product empirical spectral distribution and normalized action satisfy, almost surely,
\begin{equation}\label{eq:general-free-product-law}
F^{\mathcal R_r(M_{0,p},M_{1,p})}
\Rightarrow \rho_1\boxtimes\rho_0,
\end{equation}
and
\begin{equation}\label{eq:general-free-action-limit}
\begin{aligned}
\frac1p\mathfrak A_{\tau}^{(2)}(M_{0,p},M_{1,p})
\longrightarrow{}&-\frac{\tau}{2}\sum_{i=0}^1\int\log(1-x)\,\dd\rho_i(x)\\
&+\tau\int\log(1-\sqrt x)\,\dd(\rho_1\boxtimes\rho_0)(x).
\end{aligned}
\end{equation}
\end{proposition}

\begin{proof}
Apply \cref{lem:functional-calculus-freeness} on $K_i=[0,L_i]$ with $f=g=h_r$.  The limiting ridge variables $h_r(x_0)$ and $h_r(x_1)$ are free and have laws $\rho_0$ and $\rho_1$.  Put $q_i=h_r(L_i)<1$.  The ridge spectra lie in $[0,q_i]$, so the same polynomial-approximation argument applies to the continuous map $x\mapsto\sqrt{x}$ even when the limiting lower edge is zero.  Mixed moments involving $h_r(M_{1,p})^{1/2}$ converge to the corresponding moments involving $h_r(x_1)^{1/2}$.  The normalized moments of $h_r(M_{1,p})^{1/2}h_r(M_{0,p})h_r(M_{1,p})^{1/2}$ converge to those of $h_r(x_1)^{1/2}h_r(x_0)h_r(x_1)^{1/2}$, whose law is $\rho_1\boxtimes\rho_0$.  Moreover, the finite-$p$ product spectra lie in $[0,q_0q_1]$, so compact moment convergence upgrades to the weak convergence in \eqref{eq:general-free-product-law}.  Since $q_i<1$ and $q_0q_1<1$, the functions $\log(1-x)$ and $\log(1-\sqrt x)$ are bounded and continuous on the relevant closed envelopes, including at the origin.  Applying Definition~\ref{def:two-sample-psd-extension} now gives \eqref{eq:general-free-action-limit}.
\end{proof}

No lower spectral gap is required in Proposition~\ref{prop:free-ridge-transfer}.  When the finite-$p$ matrices are positive definite, the left side of \eqref{eq:general-free-action-limit} is the genuine Gaussian covariance action by Definition~\ref{def:two-sample-psd-extension}.

\begin{corollary}[Identity-Wishart spectral limit]\label{cor:wishart-free-limit}
Under the identity-Wishart assumptions above with $c_0,c_1\in(0,\infty)$, the normalized two-sample spectral action converges almost surely as
\begin{equation}\label{eq:two-sample-limit}
\frac1p\mathfrak A_{\tau}^{(2)}(S_0,S_1)
\xrightarrow{\as}
b_\tau^{(2)}(c_0,c_1),
\end{equation}
where
\begin{equation}\label{eq:free-bias}
\begin{aligned}
b_\tau^{(2)}(c_0,c_1)
={}&-\frac{\tau}{2}\sum_{i=0}^1\int\log(1-x)\,\dd\rho_{c_i,r}(x)\\
&+\tau\int\log(1-\sqrt x)\,
\dd\bigl(\rho_{c_1,r}\boxtimes\rho_{c_0,r}\bigr)(x).
\end{aligned}
\end{equation}
Here $\boxtimes$ denotes free multiplicative convolution on $[0,\infty)$.  The limiting product law also has an atom at zero with mass
\begin{equation}\label{eq:product-zero-atom}
(\rho_{c_1,r}\boxtimes\rho_{c_0,r})(\{0\})
=\max_{i=0,1}\left(1-\frac1{c_i}\right)_+.
\end{equation}
\end{corollary}

\begin{proof}
The Marchenko--Pastur theorem and upper-edge convergence give $F^{S_i}\Rightarrow\mu_{\mathrm{MP},c_i}$ and an almost-sure common compact upper spectral envelope.  More strongly, the real-Wishart strong asymptotic-freeness result of Lodhia, Levin and Levina~\cite[Appendix~B]{LodhiaLevinLevina2022} applies to independent matrices with possibly different ratios $p/n_i\to c_i>0$ and yields strong, hence normalized mixed-moment, convergence to free Marchenko--Pastur variables.  Proposition~\ref{prop:free-ridge-transfer} therefore gives \eqref{eq:two-sample-limit}--\eqref{eq:free-bias}.  Since the ridge map fixes the origin,
$\rho_{c_i,r}(\{0\})=(1-c_i^{-1})_+$.  The standard zero-atom formula for multiplicative free convolution then gives \eqref{eq:product-zero-atom}.  Eventual positive definiteness for $c_i<1$ follows from the lower-edge limit.
\end{proof}

When $c_0,c_1<1$, the sample covariances are eventually positive definite almost surely, and \eqref{eq:two-sample-limit} is the genuine Gaussian KL-UOT covariance-action limit.  If either aspect ratio is at least one, the statement is interpreted through the positive-semidefinite spectral continuation.

The strong-convergence input in the next statement is standard; the KL-UOT-specific consequence is the identification of the limiting ridge-product spectrum, and in the symmetric case its explicit algebraic edges.

\begin{proposition}[Ridge-product spectral convergence]\label{prop:strong-ridge-product}
Under the identity-Wishart assumptions above, set
\[
\mathcal R_p:=\mathcal R_r(S_0,S_1),
\qquad
\lambda_{c_0,c_1,r}:=\rho_{c_1,r}\boxtimes\rho_{c_0,r}.
\]
Then the ridge-product spectrum converges almost surely in Hausdorff distance:
\begin{equation}\label{eq:ridge-spectrum-hausdorff}
 d_H\!\left(\operatorname{spec}(\mathcal R_p),
 \operatorname{supp}(\lambda_{c_0,c_1,r})\right)\longrightarrow0,
\end{equation}
where $d_H$ denotes Hausdorff distance on compact subsets of $\mathbb R$.  Consequently, the extreme eigenvalues converge to the endpoints of the limiting support:
\begin{equation}\label{eq:ridge-extreme-general}
\lambda_{\min}(\mathcal R_p)\to\min\operatorname{supp}(\lambda_{c_0,c_1,r}),
\qquad
\lambda_{\max}(\mathcal R_p)\to\max\operatorname{supp}(\lambda_{c_0,c_1,r}).
\end{equation}
In the symmetric nonsingular case $c_0=c_1=c\in(0,1)$, the support endpoints from Corollary~\ref{cor:ridge-product-edge} are also the almost-sure extreme-eigenvalue limits:
\begin{equation}\label{eq:symmetric-ridge-edge-convergence}
\lambda_{\min}(\mathcal R_p)\to E_-(c,r),
\qquad
\lambda_{\max}(\mathcal R_p)\to E_+(c,r)
\qquad\text{a.s.}
\end{equation}
\end{proposition}

\begin{proof}
By the real-Wishart strong asymptotic-freeness result cited in the preceding proof, $(S_0,S_1)$ converges strongly to a free positive pair $(s_0,s_1)$ with laws $\mu_{\mathrm{MP},c_0}$ and $\mu_{\mathrm{MP},c_1}$.  Strong convergence is stable under continuous functional calculus on the almost-sure common compact spectral envelope; polynomial approximation therefore transfers strong convergence through $h_r$ and the principal square root.  Hence
\[
\mathcal R_p\longrightarrow
h_r(s_1)^{1/2}h_r(s_0)h_r(s_1)^{1/2}
\]
strongly.  The limit has law $\lambda_{c_0,c_1,r}$.  We now apply \cite[Proposition~2.1]{CollinsMale2014} to the already established one-variable strongly convergent self-adjoint sequence $\mathcal R_p$.  That proposition identifies strong convergence with weak convergence of the spectral measures together with Hausdorff convergence of the spectra.  Since the limiting variable is realized in a faithful tracial $C^*$-probability space, its spectrum is the support of its spectral law $\lambda_{c_0,c_1,r}$.  This gives \eqref{eq:ridge-spectrum-hausdorff}, and continuity of the minimum and maximum under Hausdorff convergence of nonempty compact subsets of $\mathbb R$ gives \eqref{eq:ridge-extreme-general}.  The symmetric statement follows from Corollary~\ref{cor:ridge-product-edge}.
\end{proof}

\begin{theorem}[Random-orientation ridge-product limit]\label{thm:random-orientation-free}
For $i=0,1$, let $D_{i,p}\succ0$ be deterministic with empirical laws $H_{i,p}\Rightarrow H_i$ and spectra contained in a fixed interval $[\underline\lambda_i,\overline\lambda_i]\Subset(0,\infty)$.  Let $X_{i,p}\in\mathbb R^{n_i\times p}$ have independent standard Gaussian entries with $p/n_i\to c_i\in(0,\infty)$, let $U_{i,p}$ be independent Haar orthogonal matrices independent of the $X_{i,p}$, and define
\[
T_{i,p}:=D_{i,p}^{1/2}\frac{X_{i,p}^{\top}X_{i,p}}{n_i}D_{i,p}^{1/2},
\qquad
S_{i,p}:=U_{i,p}T_{i,p}U_{i,p}^{\top}.
\]
Put
\[
\rho_i^{\rm def}:=(h_r)_\#F_{c_i,H_i}.
\]
Then the product empirical law and normalized action satisfy, almost surely,
\begin{equation}\label{eq:random-orientation-product-law}
F^{\mathcal R_r(S_{0,p},S_{1,p})}
\Rightarrow
\rho_1^{\rm def}\boxtimes\rho_0^{\rm def},
\end{equation}
and
\begin{equation}\label{eq:random-orientation-action}
\begin{aligned}
\frac1p\mathfrak A_{\tau}^{(2)}(S_{0,p},S_{1,p})
\longrightarrow{}&-\frac{\tau}{2}\sum_{i=0}^1
\int\log(1-x)\,\dd\rho_i^{\rm def}(x)\\
&+\tau\int\log(1-\sqrt x)\,
\dd(\rho_1^{\rm def}\boxtimes\rho_0^{\rm def})(x).
\end{aligned}
\end{equation}
\end{theorem}

\begin{proof}
The deformed Marchenko--Pastur theorem gives $F^{T_{i,p}}\Rightarrow F_{c_i,H_i}$ almost surely, and the uniform population spectral bounds together with the Wishart upper edge give $\sup_p\|T_{i,p}\|_{\rm op}<\infty$ almost surely after discarding finitely many indices.  Let $\Omega_T$ be the probability-one event on which both of these assertions hold for $i=0,1$.  Fix $\omega\in\Omega_T$ and regard the two sequences $T_{0,p}(\omega),T_{1,p}(\omega)$ as deterministic.  Since a simultaneous conjugation does not affect mixed normalized traces, the pair $(S_{0,p},S_{1,p})$ is equivalent for this purpose to
\[
T_{0,p}(\omega),
\qquad
V_pT_{1,p}(\omega)V_p^\top,
\qquad
V_p:=U_{0,p}^\top U_{1,p},
\]
where $V_p$ is Haar orthogonal and independent of the fixed matrices.  The weak convergence of the empirical laws together with the common operator-norm envelope implies convergence of every normalized moment of the conditioned deterministic sequences.  Thus the hypotheses of the orthogonal-Haar asymptotic-freeness theorem are met, and for this fixed $\omega$ that theorem gives conditional probability one over the Haar variables for convergence to a free pair with laws $F_{c_0,H_0}$ and $F_{c_1,H_1}$; see, for example, \cite[Proposition~2.9]{CollinsHayase2023} for this first-order formulation.  Applying Proposition~\ref{prop:free-ridge-transfer} on that conditional probability-one event yields \eqref{eq:random-orientation-product-law} and \eqref{eq:random-orientation-action}.  Thus, if $E$ denotes the desired joint convergence event,
\[
\Pr_U(E\mid T)(\omega)=1
\qquad(\omega\in\Omega_T).
\]
Since $\Pr(\Omega_T)=1$, Tonelli's theorem gives
\[
\Pr_{T,U}(E)
=\int \Pr_U(E\mid T)\,\dd\Pr_T
=1.
\]
When $c_i<1$, the lower-edge bound and $D_{i,p}\succeq\underline\lambda_iI$ imply eventual positive definiteness.
\end{proof}

The construction is equivalent in law to Gaussian sample covariance matrices with randomly oriented population covariances $U_{i,p}D_{i,p}U_{i,p}^{\top}$.  If $c_0,c_1<1$, eventual positive definiteness identifies \eqref{eq:random-orientation-action} with the corresponding Gaussian KL-UOT covariance-action limit.

We next specialize the product law to its small-aspect-ratio behavior and, in the symmetric Wishart model, to an explicit analytic description used in Section~\ref{sec:numerics}.

\begin{proposition}[Small-aspect-ratio expansion]\label{prop:two-sample-small-c}
For fixed $\tau>0$, the two-sample limit admits the expansion, as $c_0,c_1\downarrow0$,
\begin{equation}\label{eq:two-sample-small-c}
b_\tau^{(2)}(c_0,c_1)
=\frac{c_0+c_1}{4(1+2/\tau)}
+O\!\left((\sqrt{c_0}+\sqrt{c_1})(c_0+c_1)\right).
\end{equation}
\end{proposition}

\begin{proof}
Fix $c_0,c_1\in(0,1)$ and first take $p,n_0,n_1\to\infty$ along a sequence with $p/n_i\to c_i$.  Write $E_i=S_i-I$.  The Marchenko--Pastur edge limits give
\[
\limsup_{p\to\infty}\|E_i\|_{\rm op}
\le 2\sqrt{c_i}+c_i
\qquad\text{a.s.}
\]
For every $\eta>0$, the edge bound implies, almost surely for all sufficiently large $p$,
\[
\|E_i\|_{\rm op}\le 2\sqrt{c_i}+c_i+\eta.
\]
Set
\[
\delta_\eta(c_0,c_1)=\max_{i=0,1}\{2\sqrt{c_i}+c_i+\eta\}.
\]
For $c_0,c_1$ and $\eta$ sufficiently small, the dimension-uniform expansion in \cref{lem:uniform-local-covariance} applies and yields
\[
\frac1p\Acal_{\rm cov,\tau}(S_0,S_1)
=\frac{1}{4(1+2/\tau)}\frac1p\tr(E_0-E_1)^2
+O\!\left(\delta_\eta(c_0,c_1)^3\right),
\]
with a constant depending only on $\tau$.  The Marchenko--Pastur second centered moment gives
\[
\frac1p\tr E_i^2\longrightarrow c_i,
\]
whereas first-order asymptotic freeness and $p^{-1}\tr E_i\to0$ imply
\[
\frac1p\tr(E_0E_1)\longrightarrow0.
\]
Taking the high-dimensional limit and invoking \cref{cor:wishart-free-limit} gives
\[
b_\tau^{(2)}(c_0,c_1)
=\frac{c_0+c_1}{4(1+2/\tau)}
+O\!\left(\delta_\eta(c_0,c_1)^3\right).
\]
Letting $\eta\downarrow0$ and then $c_0,c_1\downarrow0$ yields
\[
\delta_0(c_0,c_1)^3
=O\!\left((\sqrt{c_0}+\sqrt{c_1})(c_0+c_1)\right),
\]
which proves \eqref{eq:two-sample-small-c}.  The order of limits is explicit: first $p,n_i\to\infty$ at fixed $c_i$, then $c_i\downarrow0$.
\end{proof}

Thus the leading term in \eqref{eq:two-sample-small-c} is the sum of the two one-sample quadratic contributions.

\subsection{Subordination representation}

For the explicit density, algebraic-branch, and regularity analysis in the remainder of this section we return to the nonsingular regime $c_i\in(0,1)$.  The all-aspect-ratio weak and strong spectral limits above remain valid independently of this restriction.

For numerical evaluation of the free multiplicative convolution in \cref{cor:wishart-free-limit}, we use analytic subordination.  For a probability measure $\rho$ on $[0,\infty)$, define on $\mathbb C\setminus\mathbb R_+$
\begin{equation}\label{eq:eta-transform}
\psi_\rho(z)=\int\frac{zt}{1-zt}\,\dd\rho(t),
\qquad
\eta_\rho(z)=\frac{\psi_\rho(z)}{1+\psi_\rho(z)}.
\end{equation}

For the ridge-transformed Marchenko--Pastur law, \eqref{eq:eta-transform} can be evaluated directly from the usual Marchenko--Pastur Stieltjes transform.  Let
\[
m_{\mathrm{MP},c}(z):=\int\frac{1}{x-z}\,\dd\mu_{\mathrm{MP},c}(x),
\qquad z\in\mathbb C\setminus[a_c,b_c],
\]
with the branch normalized by $m_{\mathrm{MP},c}(z)\sim-1/z$ at infinity.

\begin{lemma}[Ridge-transformed Marchenko--Pastur law]\label{lem:ridge-mp-eta}
Fix $c\in(0,1)$.  Let $r=2/\tau$ and $\rho_{c,r}=(h_r)_\#\mu_{\mathrm{MP},c}$.  With
\[
a_c=(1-\sqrt c)^2,\qquad b_c=(1+\sqrt c)^2,
\]
define
\begin{equation}\label{eq:ridge-support-endpoints}
\alpha_{c,r}:=\frac{ra_c}{1+ra_c},\qquad
\beta_{c,r}:=\frac{rb_c}{1+rb_c}.
\end{equation}
These endpoints satisfy $0<\alpha_{c,r}<\beta_{c,r}<1$, and the transformed law has density
\begin{equation}\label{eq:ridge-mp-density}
\frac{\dd\rho_{c,r}}{\dd t}(t)
=
\frac{\sqrt{(1+ra_c)(1+rb_c)}}{2\pi c r}
\frac{\sqrt{(t-\alpha_{c,r})(\beta_{c,r}-t)}}{t(1-t)^2}
\mathbf 1_{[\alpha_{c,r},\beta_{c,r}]}(t).
\end{equation}
Define also
\begin{equation}\label{eq:J-cr}
J_{c,r}(z):=
\frac{1}{r(1-z)}
 m_{\mathrm{MP},c}\!\left(-\frac{1}{r(1-z)}\right),
\end{equation}
The corresponding $\eta$-transform is
\begin{equation}\label{eq:eta-ridge-explicit}
\eta_{\rho_{c,r}}(z)
=\frac{z[1-J_{c,r}(z)]}{1-zJ_{c,r}(z)}.
\end{equation}
Its first moment is
\begin{equation}\label{eq:ridge-first-moment}
m_1(\rho_{c,r})
=1-\frac1r m_{\mathrm{MP},c}(-1/r).
\end{equation}
The formula in \eqref{eq:eta-ridge-explicit} is understood on $\mathbb C\setminus\mathbb R_+$ by analytic continuation; the apparent singularity at $z=1$ is removable wherever the defining $\eta$-transform is finite.
\end{lemma}

\begin{proof}
The inverse ridge map is $x=t/[r(1-t)]$ and $\dd x/\dd t=1/[r(1-t)^2]$.  Substituting this change of variables into the Marchenko--Pastur density gives \eqref{eq:ridge-support-endpoints}--\eqref{eq:ridge-mp-density}.  If $X\sim\mu_{\mathrm{MP},c}$ and $T=h_r(X)=rX/(1+rX)$, then
\[
\psi_{\rho_{c,r}}(z)
=\int\frac{zrX}{1+r(1-z)X}\,\dd\mu_{\mathrm{MP},c}(X)
=\frac{z}{1-z}[1-J_{c,r}(z)].
\]
Substitution into $\eta=\psi/(1+\psi)$ gives \eqref{eq:eta-ridge-explicit}; the first moment is $\int h_r(x)\,\dd\mu_{\mathrm{MP},c}(x)=1-r^{-1}m_{\mathrm{MP},c}(-1/r)$.
\end{proof}

Because $m_{\mathrm{MP},c}$ has the standard algebraic square-root form, \eqref{eq:eta-ridge-explicit} is explicit once the usual Stieltjes-transform branch is selected.  Numerical contour choices are discussed in Remark~\ref{rem:subordination-numerics}.

\begin{remark}\label{rem:algebraicity-priority}
Algebraic positive laws are classically closed under free multiplicative convolution; see the polynomial method of Rao and Edelman~\cite{RaoEdelman2008}.  The next proposition does not claim algebraicity of $\rho_{c,r}\boxtimes\rho_{c,r}$ as such.  What is specific to KL-UOT is the explicit low-degree polynomial obtained after eliminating the Marchenko--Pastur and subordination variables, together with the normalization that selects the physical branch entering the transport functional.
\end{remark}

\begin{proposition}[Symmetric ridge-product equation]\label{prop:ridge-product-algebraic}
Fix $c\in(0,1)$ and $r>0$, let $\rho=\rho_{c,r}$, $\lambda=\rho\boxtimes\rho$, and write
\[
Y(z):=\eta_\lambda(z),\qquad
D_{c,r}:=1+r-cr,\qquad E_{c,r}:=1+r+cr.
\]
The physical analytic branch $Y$ is characterized by the algebraic relation
\begin{equation}\label{eq:ridge-product-algebraic}
P_{c,r}(z,Y(z))=0,
\end{equation}
where
\begin{equation}\label{eq:P-cr}
P_{c,r}(z,y)
:=zy(D_{c,r}y-E_{c,r})^2
-r^2\bigl[z\{1+(c-1)y\}+cy\bigr]^2.
\end{equation}
The branch relevant to the free product is selected by
\begin{equation}\label{eq:Y-origin-normalization}
Y(z)=m_1(\rho)^2z+O(z^2),\qquad z\to0,
\end{equation}
together with the Pick normalization in Proposition~\ref{prop:subordination}.
\end{proposition}

\begin{proof}
For the Marchenko--Pastur convention used here,
\begin{equation}\label{eq:mp-stieltjes-quadratic}
c\xi m^2+(\xi+c-1)m+1=0,
\qquad m=m_{\mathrm{MP},c}(\xi).
\end{equation}
Let $w$ lie near the origin away from zero and put $y=\eta_\rho(w)$.  From \eqref{eq:eta-ridge-explicit},
\[
J_{c,r}(w)=\frac{w-y}{w(1-y)},
\qquad
m_{\mathrm{MP},c}\!\left(-\frac{1}{r(1-w)}\right)
=r(1-w)J_{c,r}(w).
\]
Substitution into \eqref{eq:mp-stieltjes-quadratic} and simplification give
\begin{equation}\label{eq:rho-eta-polynomial}
r w^2\{1+(c-1)y\}+cr y^2
+w y(D_{c,r}y-E_{c,r})=0.
\end{equation}
Substitute the symmetric subordination relations $y=Y(z)=\eta_\rho(w)$ and $w^2=zy$.  On a punctured neighborhood of zero, \eqref{eq:rho-eta-polynomial} becomes
\[
w(D_{c,r}y-E_{c,r})
=-r\bigl[z\{1+(c-1)y\}+cy\bigr].
\]
Squaring and substituting $w^2=zy$ yields \eqref{eq:ridge-product-algebraic}; analyticity extends the identity to the normalized branch.  Since first moments multiply under freeness, $m_1(\lambda)=m_1(\rho)^2$, which gives \eqref{eq:Y-origin-normalization}.
\end{proof}

Equation~\eqref{eq:ridge-product-algebraic} is a necessary algebraic relation.  The elimination step also creates extraneous algebraic branches, which are excluded by the origin and Pick normalizations.

\begin{corollary}[Support and edge regularity]\label{cor:ridge-product-edge}
For $c\in(0,1)$ and $r>0$, the factor law $\rho_{c,r}$ is a compactly supported Jacobi measure on $(0,1)$ with exponent $1/2$ at both endpoints.  Its free multiplicative square has a single support interval: there exist $0<E_-(c,r)<E_+(c,r)<1$ such that
\begin{equation}\label{eq:product-support-interval}
\operatorname{supp}(\rho_{c,r}\boxtimes\rho_{c,r})=[E_-(c,r),E_+(c,r)],
\end{equation}
and the product density has square-root decay at both $E_-(c,r)$ and $E_+(c,r)$.
\end{corollary}

\begin{proof}
By \eqref{eq:ridge-mp-density}, the density of $\rho_{c,r}$ can be written on $[\alpha_{c,r},\beta_{c,r}]\Subset(0,1)$ as
\[
\frac{\dd\rho_{c,r}}{\dd t}(t)
=w_{c,r}(t)(t-\alpha_{c,r})^{1/2}(\beta_{c,r}-t)^{1/2},
\]
where $w_{c,r}$ is smooth and bounded above and below by strictly positive constants on the support.  Hence the two Jacobi endpoint exponents are
\[
t_-=t_+=\frac12\in(-1,1),
\]
the support is compact and separated from the origin, and the density is strictly positive in its interior.  These are precisely the compact-Jacobi hypotheses needed below.  To match the mean-one normalization in Ji's published theorem, let $m=m_1(\rho_{c,r})>0$ and write $D_s\mu$ for the push-forward of $\mu$ under $x\mapsto sx$.  Then $\widetilde\rho:=D_{1/m}\rho_{c,r}$ has mean one; positive dilation preserves compactness, interior positivity, and the endpoint exponents $1/2$.  Ji's published Jacobi theorem~\cite[Theorem~3.3]{Ji2021} therefore applies and gives a single compact support interval and square-root edge decay for $\widetilde\rho\boxtimes\widetilde\rho$.  Since positive dilations satisfy
\[
(D_a\mu)\boxtimes(D_b\nu)=D_{ab}(\mu\boxtimes\nu),
\]
we have $\widetilde\rho\boxtimes\widetilde\rho=D_{1/m^2}(\rho_{c,r}\boxtimes\rho_{c,r})$.  Dilating back preserves the Jacobi exponents and proves \eqref{eq:product-support-interval}.  More explicitly, if $U,V$ are free positive variables with law $\rho_{c,r}$, then
\[
\alpha_{c,r}^2 I\preceq V^{1/2}UV^{1/2}\preceq\beta_{c,r}^2 I,
\]
because $\alpha_{c,r}I\preceq U,V\preceq\beta_{c,r}I$.  This yields
\[
\alpha_{c,r}^2\le E_-(c,r)<E_+(c,r)\le\beta_{c,r}^2<1,
\]
which makes the support separation from both $0$ and $1$ explicit.
\end{proof}

\begin{remark}\label{rem:algebraic-edges}
Equation~\eqref{eq:ridge-product-algebraic} converts the symmetric subordination problem into a single algebraic curve.  In particular, finite positive boundary branch points of the algebraic continuation of the normalized physical solution $Y$, away from removable singularities, are contained in the discriminant set
\[
P_{c,r}(z,y)=0,
\qquad
\partial_yP_{c,r}(z,y)=0.
\]
Since $G_\lambda(\zeta)=\{\zeta[1-Y(1/\zeta)]\}^{-1}$, reciprocals of those physical boundary branch points provide candidate support edges.  Corollary~\ref{cor:ridge-product-edge} supplies the square-root edge behavior; in general we do not claim a closed-form ordering of all algebraic roots, and numerical edge selection must retain the normalized Pick branch.

For the symmetric parameter choice used in Figure~\ref{fig:two-sample-density-support}, $(c,r)=(1/2,1)$, the discriminant can be identified completely.  Direct calculation gives
\begin{equation}\label{eq:figure2-discriminant}
\operatorname{Disc}_y P_{1/2,1}(z,y)
=\frac{z^2(z+5)^2}{256}
\bigl(8z^3-431z^2+934z+17\bigr).
\end{equation}
Let $q(z)=8z^3-431z^2+934z+17$.  Since $q(-1)<0<q(0)$, $q(0)>0>q(3)$, and $q(50)<0<q(60)$, the cubic has one real root in each of $(-1,0)$, $(0,3)$, and $(50,60)$; these exhaust its three roots.  The discriminant has exactly two finite positive nonzero candidate branch points,
\[
z_-\simeq2.2809134166,
\qquad
z_+\simeq51.6121374596.
\]
Corollary~\ref{cor:ridge-product-edge} also shows that the physical product law has a single compact support interval with two square-root edges.  At each edge the Cauchy transform, and hence $Y$ through \eqref{eq:cauchy-from-eta}, has the corresponding square-root boundary branch point at the reciprocal spectral coordinate.  The normalized physical branch must therefore have exactly two finite positive boundary branch points.  Since \eqref{eq:figure2-discriminant} has no other positive nonzero candidates, they must be $z_-$ and $z_+$, and the Cauchy--$\eta$ relation gives
\begin{equation}\label{eq:figure2-support-edges}
E_-=z_+^{-1}\simeq0.0193752875,
\qquad
E_+=z_-^{-1}\simeq0.4384208505.
\end{equation}
\end{remark}

The subordination statement below is the specialization of Biane's multiplicative subordination theorem on $\mathbb R_+$~\cite[Theorem~3.6]{Biane1998}, together with the equivalent two-sided $\eta$-transform formulation and analytic continuation for compactly supported laws~\cite{BercoviciVoiculescu1993,BelinschiBercovici2007}.

\begin{proposition}[Ridge-product subordination]\label{prop:subordination}
Let $\mu=\rho_{c_0,r}$, $\nu=\rho_{c_1,r}$ and $\lambda=\mu\boxtimes\nu$, where $c_i\in(0,1)$ and $r>0$.  Then there exists a unique normalized pair $\omega_0,\omega_1:\mathbb C\setminus\mathbb R_+\to\mathbb C\setminus\mathbb R_+$ such that, for $i=0,1$,
\[
\omega_i(\bar z)=\overline{\omega_i(z)},\qquad
\omega_i(\mathbb C^+)\subset\mathbb C^+,\qquad
\arg\omega_i(z)\ge\arg z\quad(z\in\mathbb C^+),
\]
where $\arg$ is the principal argument on $\mathbb C^+$.  Each $\omega_i$ extends analytically to a neighborhood of $0$, with $\omega_i(0)=0$ and
\[
\omega_0'(0)=m_1(\nu),\qquad
\omega_1'(0)=m_1(\mu).
\]
Equivalently, $\omega_0(z)/z\to m_1(\nu)$ and $\omega_1(z)/z\to m_1(\mu)$ as $z\to0$.  The pair satisfies
\begin{equation}\label{eq:subordination-identities}
\eta_\lambda(z)=\eta_\mu(\omega_0(z))
=\eta_\nu(\omega_1(z)),
\qquad
\omega_0(z)\omega_1(z)=z\eta_\lambda(z).
\end{equation}
If $G_\lambda(\zeta)=\int(\zeta-t)^{-1}\,\dd\lambda(t)$, then
\begin{equation}\label{eq:cauchy-from-eta}
G_\lambda(\zeta)
=\frac{1}{\zeta\,[1-\eta_\lambda(1/\zeta)]},
\qquad \zeta\in\mathbb C^+.
\end{equation}
Moreover, $\lambda$ is purely absolutely continuous, and its density is recovered almost everywhere by Stieltjes inversion.  In the symmetric case $c_0=c_1=c$, uniqueness gives $\omega_0=\omega_1=\omega$ and the system reduces to the scalar equation
\begin{equation}\label{eq:symmetric-subordination}
\omega(z)^2=z\,\eta_{\rho_{c,r}}(\omega(z)),
\qquad
\eta_{\rho_{c,r}\boxtimes\rho_{c,r}}(z)
=\eta_{\rho_{c,r}}(\omega(z)).
\end{equation}
The relevant solution is the analytic branch satisfying
\begin{equation}\label{eq:subordination-branch}
\omega(z)=m_1(\rho_{c,r})z+O(z^2),\qquad z\to0.
\end{equation}
\end{proposition}

\begin{proof}
The assumptions on $\mu$ and $\nu$ place them in the nondegenerate setting of probability measures on $\mathbb R_+$ in Biane's multiplicative subordination theorem~\cite[Theorem~3.6]{Biane1998}.  Applying that result with the two factors interchanged gives the unique one-sided maps $\omega_0$ and $\omega_1$ with the displayed half-plane and argument conditions; Schwarz reflection gives the conjugate continuation to the slit plane.  Near the origin, all three $\eta$-transforms are analytic and have nonzero linear terms.  Local inversion gives $\omega_0=\eta_\mu^{-1}\circ\eta_\lambda$ and $\omega_1=\eta_\nu^{-1}\circ\eta_\lambda$; using $m_1(\lambda)=m_1(\mu)m_1(\nu)$ then gives $\omega_0(z)/z\to m_1(\nu)$ and $\omega_1(z)/z\to m_1(\mu)$.  The standard multiplicative $\Sigma$-transform identity~\cite{BercoviciVoiculescu1993,BelinschiBercovici2007}, with $\Sigma_\rho(y)=\eta_\rho^{-1}(y)/y$, gives $\omega_0(z)\omega_1(z)=z\eta_\lambda(z)$ near zero; analytic continuation yields \eqref{eq:subordination-identities} on the slit domain.  To prove absolute continuity, note that the Marchenko--Pastur factors are atomless on their positive supports and the strictly increasing ridge map preserves atomlessness, so $\mu$ and $\nu$ have no positive atoms and satisfy $\mu(\{0\})=\nu(\{0\})=0$.  Belinschi's multiplicative positive-atom criterion~\cite{BelinschiAtoms2003} rules out atoms of $\lambda$ on $(0,\infty)$, while the zero-atom formula gives
\[
\lambda(\{0\})=\max\{\mu(\{0\}),\nu(\{0\})\}=0.
\]
Ji's regularity theorem on $\mathbb R_+$~\cite[Theorem~3.2]{Ji2021} gives $\lambda_{\rm sc}=0$.  Hence $\lambda=\lambda_{\rm ac}$, which justifies the Stieltjes-inversion statement above.  The relation \eqref{eq:cauchy-from-eta} follows from
$1+\psi_\lambda(z)=\int(1-zt)^{-1}\,\dd\lambda(t)$ after setting $z=1/\zeta$.  When $\mu=\nu$, symmetry and uniqueness of the normalized subordination pair imply $\omega_0=\omega_1$, and the coupling identity in \eqref{eq:subordination-identities} gives \eqref{eq:symmetric-subordination}.
\end{proof}

The Stieltjes-inversion formula reduces the final integral in \eqref{eq:free-bias} to one-dimensional quadrature.

\begin{remark}[Numerical implementation]\label{rem:subordination-numerics}
By \cref{lem:ridge-mp-eta}, $\eta_{\rho_{c,r}}$ is evaluated directly from the Marchenko--Pastur Stieltjes transform, without an inner spectral quadrature.  To remove the algebraic zero root of the raw symmetric equation $\omega^2=z\eta(\omega)$, the implementation uses the analytic quotient
\[
q_\rho(w):=\begin{cases}\eta_\rho(w)/w,&w\ne0,\\ m_1(\rho),&w=0,\end{cases}
\]
and solves the desingularized normalized equation $\omega=zq_\rho(\omega)$.  For asymmetric aspect ratios, we solve the analogous system $\omega_0=zq_\nu(\omega_1)$ and $\omega_1=zq_\mu(\omega_0)$.  Continuation begins at a small point on the ray from $0$ to the first target $z$, initialized by \eqref{eq:subordination-branch}.  Each subsequent spectral point uses the preceding branch value as the corrector initial value.  The implementation also checks the sign of the imaginary part required by the Pick condition, records residuals for both the desingularized and original identities, and verifies the prescribed first-derivative normalization at small $z$.  The reconstructed density is then obtained from \eqref{eq:cauchy-from-eta} by Stieltjes inversion.  The $\varepsilon\downarrow0$ boundary extrapolation used in Section~\ref{sec:numerics} remains a numerical device rather than a proved error bound.  For comparison, contour-integral algorithms for free convolution with quantitative spectral discretization are developed in~\cite{CortinovisYing2025}.
\end{remark}

\begin{remark}[Marginal spectral non-identifiability]\label{rem:nonidentifiability}
At fixed $\tau$, the marginal eigenvalue distributions of $R_0$ and $R_1$ do not determine the general two-population action.  For example, take
\[
R_0=\diag(x,y),\qquad
R_1^{(A)}=\diag(x,y),\qquad
R_1^{(B)}=\diag(y,x),
\]
with $0<x\ne y<1$.  Both choices of $R_1$ have the same spectrum, but
\[
\spec(R_1^{(A)}R_0)=\{x^2,y^2\},\qquad
\spec(R_1^{(B)}R_0)=\{xy,xy\}.
\]
The joint log term is respectively
\[
\log(1-x)+\log(1-y)
\quad\text{and}\quad
2\log(1-\sqrt{xy}),
\]
and these differ because
\[
(1-\sqrt{xy})^2-(1-x)(1-y)=(\sqrt x-\sqrt y)^2>0.
\]
A general fixed-$\tau$ estimator requires the joint product spectrum, information about the relative eigenvectors, or an asymptotic-freeness model.
\end{remark}

\section{Penalty-scaling regimes}\label{sec:phase}

Allowing the marginal penalties to vary with dimension exposes a second asymptotic scale.  Through \eqref{eq:optimal-mass}, the optimal mass determines when the raw KL-UOT value leaves its fixed-penalty saturation regime.

Let the dimension be $p$, let $a_p,b_p>0$ denote the two input masses, and allow $\tau_{i,p}$ to depend on $p$.  Write $\Acal_{*,p}$ for the corresponding optimal Gaussian shape action and set
\[
T_p=\tau_{0,p}+\tau_{1,p},\qquad
G_p=a_p^{\tau_{0,p}/T_p}b_p^{\tau_{1,p}/T_p}.
\]

The relative growth rates of the action and the total penalty determine the limiting transported mass.  Suppose that
\[
\Acal_{*,p}=K p^\beta(1+o(1)),\qquad
T_p=\bar T p^\alpha(1+o(1)),
\]
for constants $K,\bar T>0$ and exponents $\alpha,\beta\in\mathbb R$, and assume $G_p\to G\in(0,\infty)$.  The mass formula \eqref{eq:optimal-mass} gives
\begin{equation}\label{eq:phase-mass}
\frac{M_{*,p}}{G_p}
=\exp\left\{-\frac{K}{\bar T}p^{\beta-\alpha}(1+o(1))\right\}.
\end{equation}
$M_{*,p}/G_p$ tends to $0$, $e^{-K/\bar T}$, or $1$ according as $\beta>\alpha$, $\beta=\alpha$, or $\beta<\alpha$.

For equal unit masses and equal penalties, $G_p=1$, and the value reduces to
\begin{equation}\label{eq:equal-uot-exp}
\Ucal_{\tau_p,\tau_p}=2\tau_p\left(1-e^{-\Acal_{*,p}/(2\tau_p)}\right).
\end{equation}
\subsection{Critical scaling and balanced transport}

The following expansion quantifies convergence to the balanced Gaussian cost as the common penalty grows.  For equal penalties, let $\Acal_{*,p}(\tau)$ denote the optimal Gaussian shape action in dimension $p$; along a sequence $\tau_p$, write simply $\Acal_{*,p}=\Acal_{*,p}(\tau_p)$.

\begin{assumption}[Spectral boundedness]\label{ass:spectral-bounds}
There exist constants $0<\underline\lambda<\overline\lambda<\infty$ such that, for all $p$,
\[
\underline\lambda I\preceq\Sigma_{i,p}\preceq\overline\lambda I,
\qquad i=0,1.
\]
\end{assumption}

\begin{lemma}[Large-penalty expansion]\label{lem:large-tau}
Under Assumption~\ref{ass:spectral-bounds}, for equal penalties $\tau$, the covariance action satisfies, uniformly over $p$ and all admissible covariance pairs as $\tau\to\infty$,
\begin{equation}\label{eq:large-tau-cov}
\Acal_{\rm cov,\tau}(\Sigma_{0,p},\Sigma_{1,p})
=d_B^2(\Sigma_{0,p},\Sigma_{1,p})+O(p/\tau).
\end{equation}
If the Gaussian means are not necessarily zero but additionally $\|m_{0,p}-m_{1,p}\|^2=O(p)$ uniformly in $p$, then the full action satisfies
\begin{equation}\label{eq:large-tau-full}
\Acal_{*,p}(\tau)
=W_2^2(\N(m_{0,p},\Sigma_{0,p}),\N(m_{1,p},\Sigma_{1,p}))+O(p/\tau).
\end{equation}
The constants in the remainders depend only on $\underline\lambda,\overline\lambda$ and, in the second statement, on the implicit constant in $\|m_{0,p}-m_{1,p}\|^2=O(p)$.
\end{lemma}

\begin{proof}
Suppress the dimension index $p$ in the covariance and mean notation, and set $r=2/\tau$.  The uniform ridge and matrix-square-root expansions in \cref{lem:large-tau-technical} give
\[
R_i=r\Sigma_i+O_{\op}(r^2),\qquad
\mathcal R^{1/2}=r(\Sigma_1^{1/2}\Sigma_0\Sigma_1^{1/2})^{1/2}+O_{\op}(r^2),
\]
uniformly under Assumption~\ref{ass:spectral-bounds}.  Substituting these expansions into \eqref{eq:ridge-formula} requires a uniform remainder bound for the log determinant.  If $X=X^\top$ and $\|X\|_{\op}\le q<1$, then the scalar power series gives
\begin{equation}\label{eq:logdet-uniform-bound}
\left|\log\det(I-X)+\tr X\right|
\le \frac{p\|X\|_{\op}^2}{2(1-q)}.
\end{equation}
Under Assumption~\ref{ass:spectral-bounds}, the three matrix arguments entering \eqref{eq:ridge-formula} have operator norm $O(r)$ uniformly.  Writing
\[
G=(\Sigma_1^{1/2}\Sigma_0\Sigma_1^{1/2})^{1/2},
\]
\eqref{eq:logdet-uniform-bound} and \cref{lem:large-tau-technical} give
\begin{align*}
\Acal_{\rm cov,\tau}
&=\frac{\tau}{2}\tr R_0+\frac{\tau}{2}\tr R_1
  -\tau\tr\mathcal R^{1/2}+O(\tau p r^2)\\
&=\tr\Sigma_0+\tr\Sigma_1-2\tr G+O(p/\tau)\\
&=d_B^2(\Sigma_0,\Sigma_1)+O(p/\tau),
\end{align*}
where $r=2/\tau$.  Thus the Bures--Wasserstein trace combination appears explicitly as the leading term.
For the mean term, put $d=m_0-m_1$.  Under equal penalties, \eqref{eq:mean-action} gives
\[
\Acal_{\rm mean}
=d^\top[I+r(\Sigma_0+\Sigma_1)]^{-1}d.
\]
The resolvent identity and Assumption~\ref{ass:spectral-bounds} imply, uniformly on the same spectral class,
\[
[I+r(\Sigma_0+\Sigma_1)]^{-1}-I
=-r[I+r(\Sigma_0+\Sigma_1)]^{-1}(\Sigma_0+\Sigma_1),
\]
and the same identity yields
\[
\bigl\|[I+r(\Sigma_0+\Sigma_1)]^{-1}-I\bigr\|_{\op}
\le 2\overline\lambda r.
\]
This gives
\[
\left|\Acal_{\rm mean}-\|d\|^2\right|
\le 2\overline\lambda r\|d\|^2
=O(p/\tau)
\]
whenever $\|d\|^2=O(p)$.
\end{proof}

\begin{corollary}[Dense-discrepancy phase diagram]\label{cor:gaussian-phase}
Assume the conditions of Assumption~\ref{ass:spectral-bounds} and equal unit masses.  Let
\[
\mu_{i,p}=\N(m_{i,p},\Sigma_{i,p}),\qquad i=0,1,
\]
and suppose
\[
\tau_p=\bar\tau p^\alpha,
\qquad \bar\tau>0,\quad \alpha>0,
\]
with
\[
w_p:=\frac1pW_2^2(\mu_{0,p},\mu_{1,p})\to w>0.
\]
Then the normalized action converges to the balanced discrepancy,
\[
\frac{\Acal_{*,p}}p\to w
\]
and the optimal transported mass satisfies
\begin{equation}\label{eq:gaussian-phase}
M_{*,p}\longrightarrow
\begin{cases}
0,&0<\alpha<1,\\[0.5ex]
\exp\{-w/(2\bar\tau)\},&\alpha=1,\\[0.5ex]
1,&\alpha>1.
\end{cases}
\end{equation}
\end{corollary}

\begin{proof}
Since $w_p\to w<\infty$ and the covariance contribution to Gaussian $W_2^2$ is nonnegative,
\[
\|m_{0,p}-m_{1,p}\|^2\le W_2^2(\mu_{0,p},\mu_{1,p})=O(p).
\]
Hence the full version of \cref{lem:large-tau} applies.  Since $\alpha>0$, one has $\tau_p\to\infty$, and
\[
\frac{\Acal_{*,p}}p
=w_p+O(p^{-\alpha})\to w.
\]
For equal unit masses, \eqref{eq:optimal-mass} reduces to
\[
M_{*,p}=\exp\{-\Acal_{*,p}/(2\tau_p)\}.
\]
The exponent equals
\[
-\frac{w+o(1)}{2\bar\tau}p^{1-\alpha},
\]
which yields the three limits in \eqref{eq:gaussian-phase}.
\end{proof}

Thus $\tau_p\asymp p$ is the critical penalty scale for Gaussian discrepancies whose balanced $W_2^2$ cost is of order $p$; the centered covariance-only phase diagram is obtained by setting $m_{0,p}=m_{1,p}$.

\begin{remark}
The assumption $w>0$ isolates dense alternatives with balanced Gaussian Wasserstein cost of order $p$.  If $w=0$, the leading $O(p)$ discrepancy vanishes and the critical penalty scale can be smaller than $p$; determining it requires the next nonzero order of $W_2^2$ and is not covered by the phase diagram above.
\end{remark}

\begin{corollary}[Critical-scale limit]\label{cor:critical-limit}
Assume the conditions of Assumption~\ref{ass:spectral-bounds} and equal unit masses.  Let
\[
\mu_{i,p}=\N(m_{i,p},\Sigma_{i,p}),\qquad i=0,1,
\]
set $\tau_p=p\bar\tau$ with $\bar\tau>0$, and suppose
\[
w_p:=\frac1pW_2^2(\mu_{0,p},\mu_{1,p})\to w.
\]
Then the normalized KL-UOT value converges to
\begin{equation}\label{eq:critical-limit}
\frac1p\Ucal_{\tau_p,\tau_p}(\mu_{0,p},\mu_{1,p})
\longrightarrow
2\bar\tau\left(1-e^{-w/(2\bar\tau)}\right).
\end{equation}
\end{corollary}

\begin{proof}
As above, $w_p=O(1)$ implies $\|m_{0,p}-m_{1,p}\|^2=O(p)$.  By \cref{lem:large-tau}, $\Acal_{*,p}/p=w_p+O(p^{-1})\to w$.  Substitution into \eqref{eq:equal-uot-exp} gives \eqref{eq:critical-limit}.
\end{proof}

\subsection{Joint sample-noise/penalty scaling}

The preceding phase diagram starts from deterministic Gaussian discrepancies.  High-dimensional sampling noise creates a second $O(p)$ source of discrepancy even under the exact identity population.  The next result combines the sample-covariance and growing-penalty limits.

Define the balanced scalar function
\begin{equation}\label{eq:f-infinity}
f_\infty(x):=(\sqrt x-1)^2,\qquad x\ge0.
\end{equation}

\begin{lemma}\label{lem:uniform-f-balanced}
For every $L<\infty$ there exists $C_L<\infty$ such that, for all $\tau\ge1$,
\begin{equation}\label{eq:uniform-f-balanced}
\sup_{0\le x\le L}
\bigl|f_\tau(x)-f_\infty(x)\bigr|
\le \frac{C_L}{\tau}.
\end{equation}
\end{lemma}

\begin{proof}
Write $r=2/\tau$ and $t=\sqrt x$.  On a fixed compact set $0\le t\le\sqrt L$,
\[
A(r,t):=\sqrt{(1+r)(1+rt^2)}-rt
\]
is $C^2$ in $r$ near zero, uniformly in $t$, with $A(0,t)=1$ and
$\partial_rA(0,t)=(t-1)^2/2$.  Taylor expansion of $\log A(r,t)$ therefore gives
\[
\log A(r,t)=\frac r2(t-1)^2+O_L(r^2)
\]
uniformly in $t$ for $0\le r\le r_0(L)$, and hence \eqref{eq:uniform-f-balanced} for all $\tau\ge\tau_0(L):=2/r_0(L)$.  It remains only to cover the bounded interval $1\le\tau\le\tau_0(L)$.  The map
\[
(\tau,x)\longmapsto
\tau\,|f_\tau(x)-f_\infty(x)|
\]
is continuous on the compact set $[1,\tau_0(L)]\times[0,L]$, and therefore has a finite maximum $C_L^{\rm comp}$.  Enlarging the constant in the large-$\tau$ estimate to dominate $C_L^{\rm comp}$ proves \eqref{eq:uniform-f-balanced} for every $\tau\ge1$.
\end{proof}

\begin{theorem}[Joint RMT--penalty limit]\label{thm:joint-rmt-penalty}
Assume the sampling and population hypotheses of Theorem~\ref{thm:general-pop-limit}, now with $c\in(0,\infty)$, and let $\tau_p\to\infty$.  Define
\begin{equation}\label{eq:theta-infinity}
\Theta_\infty(c,H)
:=\int f_\infty(x)\,\dd F_{c,H}(x)
=\int(\sqrt x-1)^2\,\dd F_{c,H}(x).
\end{equation}
Then the normalized spectral action converges almost surely to
\begin{equation}\label{eq:joint-rmt-penalty-limit}
\frac1p\mathfrak A_{\tau_p}(S_p)
\longrightarrow \Theta_\infty(c,H).
\end{equation}
\end{theorem}

\begin{proof}
The proof of Theorem~\ref{thm:general-pop-limit} gives an almost-sure deterministic upper spectral envelope $[0,L]$ for $S_p$ after discarding finitely many indices.  Hence, by Lemma~\ref{lem:uniform-f-balanced},
\[
\left|
\frac1p\mathfrak A_{\tau_p}(S_p)
-\int f_\infty\,\dd F^{S_p}
\right|
\le\frac{C_L}{\tau_p}\longrightarrow0.
\]
The general Marchenko--Pastur theorem gives $F^{S_p}\Rightarrow F_{c,H}$ almost surely, and $f_\infty$ is continuous on $[0,L]$.  Therefore
$\int f_\infty\,\dd F^{S_p}\to\Theta_\infty(c,H)$, proving \eqref{eq:joint-rmt-penalty-limit}.
\end{proof}

\begin{corollary}[RMT-induced phase diagram]\label{cor:rmt-penalty-phase}
Assume the exact identity population $\Sigma_p=I_p$, $p/n\to c\in(0,\infty)$, and the finite-fourth-moment sampling conditions of Theorem~\ref{thm:general-pop-limit}.  Let
\begin{equation}\label{eq:b-infinity-c}
b_{\infty,c}:=\int(\sqrt x-1)^2\,\dd\mu_{\mathrm{MP},c}(x)>0
\end{equation}
and take
$\tau_p=\bar\tau p^\alpha$ with $\bar\tau>0$ and $\alpha>0$.  Define the spectral mass and raw-value proxies
\[
M_{p}^{\rm spec}
:=\exp\left\{-\frac{\mathfrak A_{\tau_p}(S_p)}{2\tau_p}\right\},
\qquad
U_{p}^{\rm spec}:=2\tau_p(1-M_p^{\rm spec}).
\]
Consequently, the spectral transported mass and normalized raw value satisfy, almost surely,
\begin{equation}\label{eq:rmt-penalty-mass-phase}
M_p^{\rm spec}\longrightarrow
\begin{cases}
0,&0<\alpha<1,\\[0.5ex]
\exp\{-b_{\infty,c}/(2\bar\tau)\},&\alpha=1,\\[0.5ex]
1,&\alpha>1,
\end{cases}
\end{equation}
and
\begin{equation}\label{eq:rmt-penalty-value-phase}
\frac{U_p^{\rm spec}}p\longrightarrow
\begin{cases}
0,&0<\alpha<1,\\[0.5ex]
2\bar\tau\bigl(1-e^{-b_{\infty,c}/(2\bar\tau)}\bigr),&\alpha=1,\\[0.5ex]
b_{\infty,c},&\alpha>1.
\end{cases}
\end{equation}
\end{corollary}

\begin{proof}
Theorem~\ref{thm:joint-rmt-penalty} with $H=\delta_1$ gives
$\mathfrak A_{\tau_p}(S_p)/p\to b_{\infty,c}$ almost surely.  Strict positivity follows because $f_\infty$ vanishes only at $1$ and $\mu_{\mathrm{MP},c}$ is nondegenerate for every $c>0$.  Therefore
\[
-\frac{\mathfrak A_{\tau_p}(S_p)}{2\tau_p}
=-\frac{b_{\infty,c}+o(1)}{2\bar\tau}
 p^{1-\alpha},
\]
which proves \eqref{eq:rmt-penalty-mass-phase}.  For $\alpha<1$, $U_p^{\rm spec}\le2\tau_p=o(p)$.  At $\alpha=1$, substitute the mass limit directly.  For $\alpha>1$, the exponent tends to zero and $1-e^{-x}=x+o(x)$, so
$U_p^{\rm spec}/p=\mathfrak A_{\tau_p}(S_p)/p+o(1)\to b_{\infty,c}$.  This proves \eqref{eq:rmt-penalty-value-phase}.
\end{proof}

Thus $\tau_p\asymp p$ is critical not only for deterministic dense Gaussian discrepancies but also for sample-covariance noise under the exact identity population.  For $c<1$, the spectral proxies coincide eventually almost surely with the genuine Gaussian plug-in transported mass and UOT value; for $c\ge1$, only their spectral-extension interpretation is asserted.

The same large-penalty expansion also transfers balanced-Wasserstein consistency to the deterministic critical KL-UOT scale.

\subsection{Critical-scale consistency transfer}

Let
\[
\mu_{i,p}=\N(m_{i,p},\Sigma_{i,p}),\qquad
w_p=\frac1pW_2^2(\mu_{0,p},\mu_{1,p}),
\]
under Assumption~\ref{ass:spectral-bounds}, and suppose $w_p=O(1)$.  At the critical scale $\tau_p=p\bar\tau$, \cref{cor:critical-limit} expresses the leading KL-UOT target as a smooth transform of the full balanced Gaussian Wasserstein cost.  Let $\widehat w_p$ be any estimator satisfying
\begin{equation}\label{eq:w-consistency}
\widehat w_p-w_p\xrightarrow{\mathbb P}0.
\end{equation}
Only consistency of $\widehat w_p$ is used; no distributional limit is required.  In the centered case, under the sampling and spectral assumptions of Tiomoko and Couillet~\cite[Corollary~1]{TiomokoCouillet2019}, their random-matrix-corrected estimator of the centered-Gaussian Wasserstein cost provides a concrete instance and satisfies the stronger almost-sure form of \eqref{eq:w-consistency}.

A finite-sample random-matrix correction of a nonnegative target can be negative, so we use the positive-part estimator
\begin{equation}\label{eq:w-positive}
\widehat w_p^+:=\max\{\widehat w_p,0\}
\end{equation}
and define the induced critical-scale estimator by
\begin{equation}\label{eq:critical-estimator}
\widehat u_p
:=2\bar\tau\left(1-e^{-\widehat w_p^+/(2\bar\tau)}\right).
\end{equation}
For the corresponding population pair, write
\begin{equation}\label{eq:critical-target}
U_p:=\Ucal_{\tau_p,\tau_p}(\mu_{0,p},\mu_{1,p}).
\end{equation}

\begin{corollary}[Critical-scale consistency transfer]\label{cor:critical-estimator}
Under Assumption~\ref{ass:spectral-bounds}, $w_p=O(1)$, $\tau_p=p\bar\tau$, and \eqref{eq:w-consistency}, the critical-scale estimator is consistent in the sense that
\begin{equation}\label{eq:u-consistency}
\widehat u_p-\frac{U_p}{p}
\xrightarrow{\mathbb P}0.
\end{equation}
If the stronger input $\widehat w_p-w_p\to0$ almost surely holds, then the convergence in \eqref{eq:u-consistency} is also almost sure.  For $g(x)=2\bar\tau(1-e^{-x/(2\bar\tau)})$, one has the dimension-free contraction bound
\begin{equation}\label{eq:lipschitz-transfer}
|g(x)-g(y)|\le|x-y|,
\qquad x,y\ge0.
\end{equation}
\end{corollary}

\begin{proof}
Since $w_p=O(1)$, the mean difference automatically satisfies $\|m_{0,p}-m_{1,p}\|^2=O(p)$.  By \cref{lem:large-tau},
\[
\frac{\Acal_{*,p}}p=w_p+O(p^{-1}).
\]
Using \eqref{eq:equal-uot-exp} then gives
\[
\frac{U_p}{p}=g(w_p)+o(1).
\]
Since $w_p\ge0$, projection onto $[0,\infty)$ is non-expansive, so
\[
|\widehat w_p^+-w_p|\le|\widehat w_p-w_p|\xrightarrow{\mathbb P}0.
\]
Since $g'(x)=e^{-x/(2\bar\tau)}\in(0,1]$ for $x\ge0$, \eqref{eq:lipschitz-transfer} follows.  Therefore
\[
\left|\widehat u_p-\frac{U_p}{p}\right|
\le|\widehat w_p^+-w_p|+o(1)
\le|\widehat w_p-w_p|+o(1)\xrightarrow{\mathbb P}0.
\]
The same deterministic inequalities give the almost-sure conclusion whenever the input consistency is almost sure.
\end{proof}

\section{Numerical illustrations}\label{sec:numerics}

The numerical study examines the finite-dimensional spectral identities, first- and second-order asymptotic predictions, penalty scaling, and the two-sample ridge-product law.  Each experiment is organized around the corresponding theoretical statement; auxiliary numerical diagnostics and complete reproduction details are provided in the Supplementary Numerical Material.

\subsection{Finite-dimensional validation}

We begin by comparing the exact spectral formulas with direct minimization of the Gaussian covariance objective at $p=2$.  For
\[
\Sigma_0=\begin{pmatrix}1.4&0.35\\0.35&0.9\end{pmatrix},\qquad
\Sigma_1=\begin{pmatrix}0.8&-0.2\\-0.2&1.6\end{pmatrix},
\]
the equal-penalty formula in Corollary~\ref{cor:ridge-product} and the asymmetric formula in Theorem~\ref{thm:exact-spectral} agree with direct Cholesky-parametrized minimization to machine precision; across the two cases the absolute discrepancy is below $3\times10^{-15}$.  Further optimizer and stationarity diagnostics are reported in the Supplementary Numerical Material.

\subsection{First-order spectral convergence}

We examine \cref{thm:general-pop-limit} for two qualitatively different population spectra,
\[
H_{\rm disc}=\tfrac12\delta_{1/2}+\tfrac12\delta_2,
\qquad
H_{\rm unif}=\operatorname{Unif}[1/2,2].
\]
At $(c,\tau)=(1/2,2)$, numerical solution of the Silverstein equation gives
\[
\Theta_2(1/2,H_{\rm disc})\simeq0.122311,
\qquad
\Theta_2(1/2,H_{\rm unif})\simeq0.089938.
\]
For $H_{\rm unif}$, the population integral is evaluated with a $48$-point Gauss--Legendre discretization, and the finite-$p$ population eigenvalues are its midpoint quantiles.  The sampling laws are Gaussian, Rademacher, and standardized Student $t_8$, namely $\sqrt{6/8}\,t_8$; the latter has unit variance and finite moments beyond order four.  For all three entry laws, the reported finite-$p$ quantity is the positive-semidefinite spectral extension $\mathfrak A_\tau(S_p)$ from \eqref{eq:psd-spectral-extension}; whenever $S_p\succ0$ it equals the nondegenerate Gaussian KL-UOT covariance action.

\begin{table}[ht]
\centering
\caption{First-order spectral convergence for $c=1/2$ and $\tau=2$, based on $140$ repetitions at $p=80$ and $70$ repetitions at $p=240$.  Entries in the three sampling columns are Monte Carlo means of $p^{-1}\mathfrak A_\tau(S_p)$, with $10^4\times\mathrm{MCSE}$ in parentheses.}
\label{tab:robustness-stress}
{\small
\begin{tabular}{@{}lrrrrr@{}}
\toprule
population law & $p$ & $\Theta_2(1/2,H)$ & Gaussian & Rademacher & std. $t_8$\\
\midrule
$H_{\rm disc}$ & 80  & 0.12231 & 0.12302 (2.14) & 0.12199 (1.31) & 0.12387 (2.36)\\
$H_{\rm disc}$ & 240 & 0.12231 & 0.12243 (1.10) & 0.12231 (0.57) & 0.12289 (1.18)\\
$H_{\rm unif}$ & 80  & 0.08994 & 0.09051 (1.84) & 0.08937 (1.41) & 0.09139 (1.84)\\
$H_{\rm unif}$ & 240 & 0.08994 & 0.09028 (0.72) & 0.08987 (0.60) & 0.09037 (0.87)\\
\bottomrule
\end{tabular}}
\end{table}

The three sampling laws approach the same deterministic limit, with visibly larger finite-size effects for the standardized $t_8$ samples at the reported dimensions.  The continuous-spectrum case confirms that the agreement is not tied to a finitely supported population law.

\subsection{Bai--Silverstein calibration}

For the identity population with $(c,\tau)=(1/2,2)$, numerical integration gives
\[
b_{2,1/2}=0.07209180,
\qquad
\mathfrak m_{1/2}^{\rm BS}(f_2)=0.07123457,
\qquad
\mathfrak v_{1/2}^{\rm BS}(f_2)=0.01789501.
\]
Table~\ref{tab:lss-convergence} reports the standardized statistic $Z_{p,n}^{(2)}$ across four dimensions.  Its empirical mean and standard deviation remain close to $0$ and $1$, respectively, while the rejection frequency is consistent with the nominal $5\%$ level.

\begin{table}[ht]
\centering
\caption{Finite-size Bai--Silverstein calibration under $H_0:\Sigma=I_p$, with $n=2p$ and $\tau=2$.  ``Size'' is the rejection probability of the one-sided nominal $5\%$ rule; its MCSE is shown in parentheses.}
\label{tab:lss-convergence}
\begin{tabular}{rrrrrr}
\toprule
$p$ & $n$ & repetitions & mean of $Z$ & sd of $Z$ & size\\
\midrule
40  & 80  & 800 & $-0.004$ & 1.012 & 0.045 (0.007)\\
80  & 160 & 600 & 0.017 & 1.009 & 0.045 (0.008)\\
160 & 320 & 350 & $-0.045$ & 0.963 & 0.043 (0.011)\\
320 & 640 & 180 & $-0.036$ & 0.995 & 0.050 (0.016)\\
\bottomrule
\end{tabular}
\end{table}

The contour and Joukowski/Fourier evaluations of the Bai--Silverstein mean and variance agree to the reported precision, providing an independent check of the calibration formulas.

The second-order calibration is specific to the real-Gaussian model.  A companion misspecification experiment with Rademacher and standardized $t_8$ entries, reported in the Supplementary Numerical Material, shows the expected size distortion when the Gaussian centering and variance are used outside their stated regime.  This is consistent with the cumulant and eigenvector corrections in general covariance-matrix LSS CLTs~\cite{NajimYao2016} and separates first-order finite-moment robustness from second-order calibration.

A representative dense-alternative power comparison with likelihood-ratio and Frobenius-type LSS benchmarks is reported in the Supplementary Numerical Material.

\subsection{Penalty-scaling phase diagram}

For the penalty-scaling illustration, take equal unit masses and
\[
\Sigma_0=I_p,\qquad
\Sigma_1=\diag(1/2,2,1/2,2,\ldots).
\]
The balanced Gaussian Wasserstein cost per coordinate is $w=0.128680$.  With $\tau_p=p^\alpha$ and $p=12800$, the transported masses for $\alpha=1/2,1,$ and $3/2$ are $0.000807$, $0.937698$, and $0.999431$, respectively.  In particular, the critical value $0.937698$ is already close to the limit $e^{-w/2}\simeq0.937686$.  The three values exhibit the subcritical mass collapse, nondegenerate critical regime, and supercritical recovery predicted by \cref{cor:gaussian-phase}; the corresponding normalized-action diagnostics are reported in the Supplementary Numerical Material.

The same critical exponent appears under the exact identity population for a different reason.  At $c=1/2$, direct Marchenko--Pastur quadrature gives $b_{\infty,1/2}\simeq0.134401$, and Corollary~\ref{cor:rmt-penalty-phase} yields, at $\bar\tau=1$,
\[
M_p^{\rm spec}\to0.935008,
\qquad
U_p^{\rm spec}/p\to0.129985.
\]
Here the $O(p)$ discrepancy is generated by sample-covariance noise rather than by a deterministic covariance mismatch, distinguishing the random-matrix phase from the preceding deterministic example.

\subsection{Two-sample ridge-product law}

For the symmetric two-sample experiment, we use
\[
c_0=c_1=1/2,\qquad \tau=2.
\]
The normalized subordination calculation yields the extrapolated estimate
\[
b_2^{(2)}(1/2,1/2)\simeq0.132896.
\]
Table~\ref{tab:two-sample-convergence} compares the normalized action and ridge-product spectrum with their limiting references.  Let $\overline A_p$ denote the Monte Carlo mean of $p^{-1}\Acal_{\rm cov,\tau}(S_0,S_1)$ and let $\widetilde\lambda_\varepsilon$ be the normalized finite-$\varepsilon$ subordination reference.  Both the action discrepancy and the mean $W_1(F^{\mathcal R_p},\widetilde\lambda_\varepsilon)$ decrease with dimension over the reported range.  Because the reference uses a fixed boundary regularization, the $W_1$ values combine finite-$p$ and numerical errors and are not interpreted as a convergence rate.

\begin{table}[ht]
\centering
\caption{Finite-size convergence of the two-sample normalized action and ridge-product ESD for $c_0=c_1=1/2$ and $\tau=2$.  Parentheses contain MCSEs; the last column reports $10^3 W_1(F^{\mathcal R_p},\widetilde\lambda_\varepsilon)$ for the fixed finite-$\varepsilon$ reference.}
\label{tab:two-sample-convergence}
{\small
\begin{tabular}{@{}rrrrrr@{}}
\toprule
$p$ & $n_0=n_1$ & reps. & $\overline A_p$ & $|\overline A_p-b_2^{(2)}|$ & $10^3 W_1$\\
\midrule
80  & 160 & 120 & 0.134212 (0.000270) & 0.001316 & 3.373 (0.096)\\
160 & 320 & 100 & 0.133468 (0.000143) & 0.000573 & 1.880 (0.048)\\
320 & 640 & 60  & 0.133294 (0.000089) & 0.000398 & 1.375 (0.040)\\
\bottomrule
\end{tabular}}
\end{table}

Figure~\ref{fig:two-sample-density-support} overlays the fixed-$\varepsilon$ subordination approximation with the pooled $p=320$ ridge-product spectrum.  For $(c,r)=(1/2,1)$, Remark~\ref{rem:algebraic-edges} identifies
\[
E_-\simeq0.019375,\qquad E_+\simeq0.438421.
\]
Proposition~\ref{prop:strong-ridge-product} identifies these as the almost-sure limits of the extreme eigenvalues; at $p=320$, their Monte Carlo means are $0.020330$ (MCSE $0.000131$) and $0.435221$ (MCSE $0.000412$).  Finite-$\varepsilon$ smoothing rounds the edge profile, so the figure diagnoses the bulk and support rather than the edge exponent.

\begin{figure}[ht]
\centering
\includegraphics[width=0.90\linewidth]{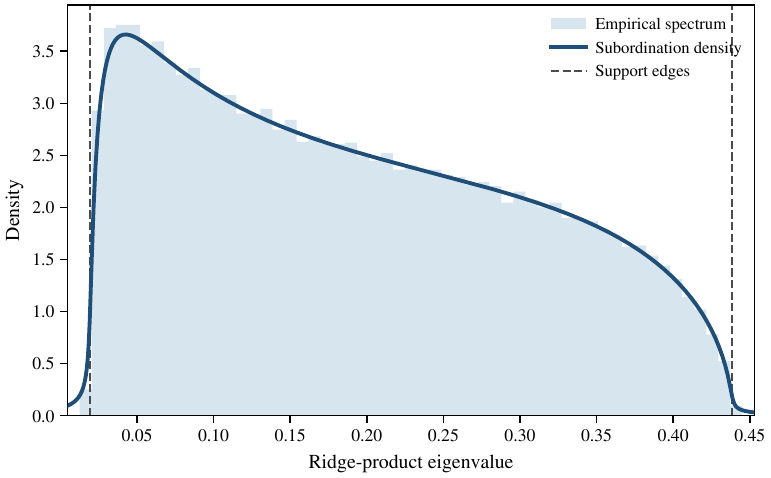}
\caption{Two-sample ridge-product spectrum and fixed-$\varepsilon$ approximation to the limiting law for $c_0=c_1=1/2$ and $\tau=2$.  The histogram shows the pooled $p=320$ empirical eigenvalues, the solid curve is the Stieltjes-inversion approximation at $\varepsilon=1.5\times10^{-3}$, and the dashed lines denote the physical support endpoints $E_-$ and $E_+$.}
\label{fig:two-sample-density-support}
\end{figure}

Independent algebraic-branch reconstruction, an asymmetric aspect-ratio check, and boundary-regularization and grid-sensitivity diagnostics are reported in the Supplementary Numerical Material.

\FloatBarrier
\section{Conclusion}\label{sec:conclusion}

Under equal penalties, Gaussian KL-UOT reduces to a nonlinear ridge-product spectral functional.  For independent real Wishart samples, this yields all-aspect free multiplicative-convolution limits and Hausdorff spectral convergence; in the symmetric nonsingular model, the algebraic physical branch identifies the support, square-root edges, and limiting extreme eigenvalues.  Independent Haar orientations extend the first-order law to deformed populations, while the one-sample problem provides an all-aspect Marchenko--Pastur benchmark with real-Gaussian Bai--Silverstein calibration for $c<1$.

The positive-semidefinite continuation separates these random-matrix limits from singular-Gaussian measure theory: it remains finite for $c_i\ge1$, whereas the original Gaussian KL-UOT interpretation is retained only for positive-definite inputs.  The mass formula also connects spectral bias to penalty scaling.  Both deterministic dense discrepancies and identity-population sample noise produce an $O(p)$ action and hence the critical scale $\tau_p\asymp p$, with the normalized action as the stable high-dimensional quantity.

Two extensions require genuinely multi-matrix methods.  General noncommuting population pairs call for block linearization, matrix Dyson equations, or operator-valued free deterministic equivalents~\cite{SpeicherVargas2012}; second-order fluctuations of the two-sample ridge-product log determinants point toward real second-order freeness~\cite{Redelmeier2014}.  Both lie beyond the scalar LSS arguments used here.

The numerical results are fully reproducible: no external data are used, and the Supplementary Numerical Material contains numerical diagnostics, the representative power study, and complete reproduction code and environment information.

\appendix

\section{Finite-dimensional Gaussian derivations}\label{app:gaussian-derivations}
The finite-dimensional Gaussian formulas in Proposition~\ref{prop:gaussian-ingredients} are recalled in the main text only as inputs to the random-matrix analysis.  For completeness, we record their derivations here.

We use the notation of Section~\ref{sec:spectral}.  Thus
\[
T=\tau_0+\tau_1,\qquad w_i=\tau_i/T,\qquad r_i=2/\tau_i,
\]
\[
\mu_i=\N(m_i,\Sigma_i),\qquad A_i=\Sigma_i^{-1},\qquad C_i=A_i+r_iI,
\]
\[
\kappa=r_1-r_0,\qquad
B=C_1^{1/2}C_0C_1^{1/2},\qquad
s(t)=\frac{\kappa+\sqrt{\kappa^2+4t}}{2},\qquad \mathsf S=s(B).
\]

We first separate the total mass from the normalized coupling.  Write a nonzero coupling as $\gamma=M\pi$, where $M>0$ and $\pi$ is a probability coupling of probability measures $\nu_0,\nu_1$.  The finite-measure relative entropy satisfies
\[
\KL(M\nu_0\mid a\mu_0)
=M\KL(\nu_0\mid\mu_0)+M\log(M/a)-M+a,
\]
and analogously for the second marginal.  For fixed shapes, the objective is therefore a strictly convex function of $M$, whose minimizer is
\[
M(\nu_0,\nu_1)
=a^{w_0}b^{w_1}
\exp\!\left\{-\frac{\Acal_{\tau_0,\tau_1}(\nu_0,\nu_1;\mu_0,\mu_1)}{T}\right\}.
\]
Substitution gives $\tau_0a+\tau_1b-TM(\nu_0,\nu_1)$.  Minimizing over the shapes yields the mass-separation formula and the optimal mass stated in Proposition~\ref{prop:gaussian-ingredients}.  The zero coupling has the strictly larger value $\tau_0a+\tau_1b$.

We next reduce the shape minimization to Gaussian marginals and verify attainment.  Let $\nu_i$ have mean $u_i$ and positive-definite covariance $P_i$, and let $g_i$ be the Gaussian with the same first two moments.  The Gelbrich inequality gives
\[
W_2^2(\nu_0,\nu_1)\ge W_2^2(g_0,g_1).
\]
Because the Gaussian references are nondegenerate, finite $\KL(\nu_i\mid\mu_i)$ implies that $\nu_i$ is absolutely continuous with respect to Lebesgue measure and therefore cannot have singular covariance.  Moreover, $\log(dg_i/d\mu_i)$ is quadratic, so equality of the first two moments gives the Pythagorean identity
\[
\KL(\nu_i\mid\mu_i)
=\KL(\nu_i\mid g_i)+\KL(g_i\mid\mu_i)
\ge \KL(g_i\mid\mu_i).
\]
Thus Gaussian moment projection cannot increase the action, and equality at a finite-objective minimizer forces $\nu_i=g_i$.

For $P\succ0$, the centered Gaussian KL divergence is
\[
\KL\!\left(\N(0,P)\mid\N(0,\Sigma)\right)
=\frac12\{\tr(\Sigma^{-1}P)-\log\det(\Sigma^{-1}P)-p\}.
\]
If $C=\Sigma^{-1/2}P\Sigma^{-1/2}$, this is one half of $\sum_j\phi(\lambda_j(C))$ with $\phi(t)=t-\log t-1$.  Since $\phi(t)\to\infty$ as $t\downarrow0$ or $t\to\infty$, finite KL sublevel sets confine every covariance eigenvalue to a compact interval bounded away from zero and infinity.  Together with the nonnegative Bures term this gives compact covariance sublevel sets after the boundary is assigned value $+\infty$, hence an interior minimizer.

The mean contribution follows from a strictly convex quadratic problem.  Its first-order conditions are
\[
u-m_0=-r_0\Sigma_0(u-v),\qquad
v-m_1=r_1\Sigma_1(u-v).
\]
With $\delta=m_0-m_1$ and $h=u-v$, subtraction yields
\[
[I+r_0\Sigma_0+r_1\Sigma_1]h=\delta.
\]
Consequently
\[
h_*=[I+r_0\Sigma_0+r_1\Sigma_1]^{-1}\delta,
\]
\[
u_*=m_0-r_0\Sigma_0h_*,\qquad
v_*=m_1+r_1\Sigma_1h_*,
\]
and direct substitution gives
\[
\Acal_{\rm mean}
=\delta^\top[I+r_0\Sigma_0+r_1\Sigma_1]^{-1}\delta.
\]

For the covariance contribution, let $L\succ0$ be the optimal Gaussian transport map, so that $Q=LPL$.  The Bures--Wasserstein differential on the positive-definite cone is
\[
D_Pd_B^2(P,Q)[H]=\tr[(I-L)H],\qquad
D_Qd_B^2(P,Q)[K]=\tr[(I-L^{-1})K].
\]
Adding the Gaussian KL derivatives at the interior minimizer gives
\[
P^{-1}=A_0+r_0(I-L),\qquad
Q^{-1}=A_1+r_1(I-L^{-1}).
\]
Using $Q^{-1}=L^{-1}P^{-1}L^{-1}$ and eliminating $P,Q$ yields
\[
LC_1L-\kappa L=C_0.
\]
Set $X=C_1^{1/2}LC_1^{1/2}$.  Then
\[
X^2-\kappa X=B.
\]
Because $B=X^2-\kappa X$ is a polynomial in the positive-definite matrix $X$, the two matrices commute.  Each scalar eigenvalue equation $x^2-\kappa x=b$ with $b>0$ has exactly one positive root,
\[
x=s(b)=\frac{\kappa+\sqrt{\kappa^2+4b}}{2}.
\]
Hence $X=s(B)=\mathsf S$ and
\[
L=C_1^{-1/2}\mathsf S C_1^{-1/2}.
\]
The positive Riccati branch uniquely determines $L$, after which the stationarity equations uniquely determine $P$ and $Q$.  This establishes the finite-dimensional Gaussian ingredients used in Section~\ref{sec:spectral}.

\section{Local and uniform expansions}\label{app:technical}

\subsection{Local Gaussian geometry}\label{app:local-geometry}
The next results describe the fixed-dimensional second-order geometry.  They are logically separate from the dimension-uniform expansion in Appendix~\ref{app:uniform-expansions}, which is derived directly from the exact ridge representation.

Put
\[
r_i=\frac{2}{\tau_i},\qquad
\mathcal L_\Sigma(X)=\Sigma X+X\Sigma,
\qquad
\mathcal K_\Sigma(X)=\Sigma X+X\Sigma+(r_0+r_1)\Sigma X\Sigma.
\]

\begin{lemma}\label{lem:local-cov-expansion}
Fix $p$ and $\Sigma\in\mathbb S_{++}^p$.  If $X,Y,K_0,K_1$ range over operator-norm-bounded subsets of the symmetric matrices, then, uniformly as $\varepsilon\to0$,
\begin{align}
d_B^2(\Sigma+\varepsilon X,\Sigma+\varepsilon Y)
&=\frac{\varepsilon^2}{2}\ip{X-Y}{\mathcal L_\Sigma^{-1}(X-Y)}+O(\varepsilon^3),\label{eq:bures-local}\\
\KL\!\left(\N(0,\Sigma+\varepsilon X)\mid\N(0,\Sigma+\varepsilon K_i)\right)
&=\frac{\varepsilon^2}{4}\tr[\Sigma^{-1}(X-K_i)\Sigma^{-1}(X-K_i)]+O(\varepsilon^3).\label{eq:kl-local}
\end{align}
The second-order covariance action is the minimum of the quadratic model obtained from \eqref{eq:bures-local}--\eqref{eq:kl-local}.
\end{lemma}

\begin{proof}
Analyticity of the principal square-root map on the positive-definite cone gives \eqref{eq:bures-local} by Fr\'echet expansion.  Taylor expansions of inversion and $\log\det$ give \eqref{eq:kl-local}.  Let $q_*$ be the minimum of the resulting quadratic model and $C_\varepsilon^*$ the exact minimized covariance objective.  The positive Riccati representation above expresses the exact minimizer through smooth operations near the identity point, so the minimizing perturbations remain bounded.  Uniform third-order remainders therefore give
\[
\varepsilon^2q_*-C\varepsilon^3\le C_\varepsilon^*\le \varepsilon^2q_*+C\varepsilon^3,
\]
which proves the last assertion.
\end{proof}

\begin{theorem}[Mass--location--shape expansion]\label{thm:local-geometry}
Fix $\mathfrak m>0$, $\bar m\in\R^p$ and $\Sigma\in\mathbb S_{++}^p$.  For $i=0,1$, let
\[
\rho_{i,\varepsilon}
=(\mathfrak m+\varepsilon\dot a_i)
\N(\bar m+\varepsilon h_i,\Sigma+\varepsilon \Delta_i),
\]
where $\Delta_i$ are symmetric and $\varepsilon$ is small enough that the masses and covariances remain admissible.  Then
\begin{align}
\Ucal_{\tau_0,\tau_1}(\rho_{0,\varepsilon},\rho_{1,\varepsilon})
={}&\varepsilon^2\frac{\tau_0\tau_1}{2\mathfrak m(\tau_0+\tau_1)}(\dot a_0-\dot a_1)^2\notag\\
&+\varepsilon^2\mathfrak m
(h_0-h_1)^\top[I+(r_0+r_1)\Sigma]^{-1}(h_0-h_1)\notag\\
&+\frac{\varepsilon^2\mathfrak m}{2}
\ip{\Delta_0-\Delta_1}{\mathcal K_\Sigma^{-1}(\Delta_0-\Delta_1)}
+o(\varepsilon^2).\label{eq:local-full}
\end{align}
If $\Sigma=\operatorname{diag}(\lambda_1,\ldots,\lambda_p)$, then the covariance term equals
\[
\frac{\varepsilon^2\mathfrak m}{2}
\sum_{j,k=1}^p
\frac{|(\Delta_0-\Delta_1)_{jk}|^2}
{\lambda_j+\lambda_k+(r_0+r_1)\lambda_j\lambda_k}
+o(\varepsilon^2).
\]
\end{theorem}

\begin{proof}
Expanding the mass-separation identity gives the first term, while the exact mean-resolvent formula gives the second.  By Lemma~\ref{lem:local-cov-expansion}, the covariance term is the minimum of
\begin{align*}
Q_\Sigma(X,Y)={}&\frac12\ip{X-Y}{\mathcal L_\Sigma^{-1}(X-Y)}
+\frac{\tau_0}{4}\tr[\Sigma^{-1}(X-\Delta_0)\Sigma^{-1}(X-\Delta_0)]\\
&+\frac{\tau_1}{4}\tr[\Sigma^{-1}(Y-\Delta_1)\Sigma^{-1}(Y-\Delta_1)].
\end{align*}
After diagonalizing $\Sigma$, each matrix entry reduces to a scalar strictly convex quadratic problem.  Its minimum is
\[
\frac12\frac{|(\Delta_0-\Delta_1)_{jk}|^2}
{\lambda_j+\lambda_k+(r_0+r_1)\lambda_j\lambda_k}.
\]
Summation gives the covariance term in \eqref{eq:local-full}.  Since the spatial action is $O(\varepsilon^2)$, the mass-separation prefactor is $\mathfrak m+o(1)$.
\end{proof}

\begin{remark}
If $\tau_0=\tau_1=\tau$ and $\Sigma=I$, then
\begin{align*}
\Ucal_{\tau,\tau}
={}&\varepsilon^2\frac{\tau}{4\mathfrak m}(\dot a_0-\dot a_1)^2
+\varepsilon^2\frac{\mathfrak m}{1+4/\tau}\|h_0-h_1\|^2\\
&+\varepsilon^2\frac{\mathfrak m}{4(1+2/\tau)}\normF{\Delta_0-\Delta_1}^2
+o(\varepsilon^2).
\end{align*}
\end{remark}

\subsection{Uniform and large-penalty expansions}\label{app:uniform-expansions}
\begin{lemma}[Uniform identity expansion]\label{lem:uniform-local-covariance}
Fix $\tau>0$ and set $r=2/\tau$.  There exist $\delta_0,C_\tau>0$, independent of $p$, such that for all symmetric $E_0,E_1$ satisfying
\[
\max_{i=0,1}\|E_i\|_{\rm op}\le\delta\le\delta_0,
\]
one has
\begin{equation}\label{eq:uniform-local-remainder}
\left|
\frac1p\Acal_{\rm cov,\tau}(I+E_0,I+E_1)
-\frac{1}{4(1+2/\tau)}\frac1p\tr(E_0-E_1)^2
\right|
\le C_\tau\delta^3.
\end{equation}
\end{lemma}

\begin{proof}
Use the exact ridge-product representation \eqref{eq:ridge-formula} and write
\[
R(E)=r(I+E)\{I+r(I+E)\}^{-1},
\qquad
P(E_0,E_1)=R(E_1)^{1/2}R(E_0)R(E_1)^{1/2}.
\]
Choose $\delta_0<1/2$.  Then $I+E_i$ has spectrum in $[1/2,3/2]$, so the spectra of $R(E_i)$ lie in a compact interval $K_R\Subset(0,1)$ depending only on $r$, and, after decreasing $\delta_0$ if necessary, the spectrum of $P(E_0,E_1)$ lies in a compact interval $K_P\Subset(0,1)$ containing the base value $(r/(1+r))^2$.

On these fixed spectral sets, inversion, the principal square root, and the functions $x\mapsto\log(1-x)$ and $x\mapsto\log(1-\sqrt{x})$ have dimension-free Fr\'echet derivative bounds through order three.  Choose contours enclosing $K_R$ and $K_P$ at positive distance from the spectra.  For holomorphic $g$ and $1\le k\le3$, writing $R_z=(zI-A)^{-1}$ and
\[
\mathcal K_k(z;H_1,\ldots,H_k)
:=\sum_{\pi\in\mathfrak S_k}
R_zH_{\pi(1)}R_z\cdots H_{\pi(k)}R_z,
\]
one has
\[
D^kg(A)[H_1,\ldots,H_k]
=\frac{1}{2\pi i}\oint_\Gamma
 g(z)\mathcal K_k(z;H_1,\ldots,H_k)\,\dd z.
\]
Contour separation gives $\sup_{z\in\Gamma}\|R_z\|_{\rm op}\le C_\Gamma$, hence
\[
\|D^kg(A)[H_1,\ldots,H_k]\|_{\rm op}
\le C_{g,k,\Gamma}\prod_{j=1}^k\|H_j\|_{\rm op},
\]
with a constant independent of $p$.  Matrix multiplication and $E\mapsto R(E)$ satisfy analogous bounds.  With the product norm
\[
\|(E_0,E_1)\|_\times:=\max\{\|E_0\|_{\rm op},\|E_1\|_{\rm op}\},
\]
the chain rule gives a constant $C_\tau'$ such that, at every base point in the chosen neighborhood, the first three Fr\'echet derivatives of
\[
F_p(E_0,E_1):=\frac1p\Acal_{\rm cov,\tau}(I+E_0,I+E_1)
\]
have multilinear operator norm at most $C_\tau'$, uniformly in $p$.  The normalized trace introduces no dimension factor, since
\[
\frac1p|\tr A|\le\|A\|_{\rm op}.
\]
After decreasing $\delta_0$ if necessary, the segment $t(E_0,E_1)$, $0\le t\le1$, stays in the same spectral neighborhood, so the same constant controls the Taylor remainder.

At the origin, $F_p(0,0)=0$ and $DF_p(0,0)=0$.  We identify the Hessian directly from the exact ridge representation.  Exchange symmetry and the identity $F_p(E,E)=0$ for all sufficiently small symmetric $E$ imply that the quadratic Taylor term depends only on $E_0-E_1$.  Setting $E_0=0$ and $E_1=tH$, \eqref{eq:identity-action-lfs} and \eqref{eq:f-local} give
\[
F_p(0,tH)
=\frac1p\sum_{j=1}^p f_\tau(1+t\lambda_j(H))
=\frac{t^2}{4(1+r)}\frac1p\tr H^2+O(t^3),
\]
where $r=2/\tau$.  Polarization therefore yields
\[
\frac12D^2F_p(0,0)[(E_0,E_1),(E_0,E_1)]
=\frac{1}{4(1+r)}\frac1p\tr(E_0-E_1)^2.
\]
This agrees with the fixed-dimensional covariance Hessian in Appendix~\ref{app:local-geometry}, but the dimension-uniform estimate is derived independently.
Taylor's theorem with integral remainder and the uniform third-derivative bound now gives
\[
\left|F_p(E_0,E_1)
-\frac{1}{4(1+r)}\frac1p\tr(E_0-E_1)^2\right|
\le C_\tau\max_i\|E_i\|_{\rm op}^3,
\]
which is \eqref{eq:uniform-local-remainder} because $r=2/\tau$.
\end{proof}

\begin{lemma}\label{lem:large-tau-technical}
Under Assumption~\ref{ass:spectral-bounds}, with $r=2/\tau$, the ridge variables admit the uniform expansions, as $r\downarrow0$ (equivalently, $\tau\to\infty$),
\[
R_i=r\Sigma_i+O_{\op}(r^2),\qquad
\mathcal R^{1/2}=r(\Sigma_1^{1/2}\Sigma_0\Sigma_1^{1/2})^{1/2}+O_{\op}(r^2).
\]
\end{lemma}

\begin{proof}
Set $\widetilde R_i=R_i/r=\Sigma_i(I+r\Sigma_i)^{-1}$.  Uniform spectral boundedness gives $\widetilde R_i=\Sigma_i+O_{\op}(r)$ and places the spectra of $\widetilde R_i$ in a fixed compact subset of $(0,\infty)$.  Therefore
\[
\frac{\mathcal R}{r^2}=\widetilde R_1^{1/2}\widetilde R_0\widetilde R_1^{1/2}
=\Sigma_1^{1/2}\Sigma_0\Sigma_1^{1/2}+O_{\op}(r).
\]
The square-root map is Lipschitz on this compact positive-definite spectral set, giving the expansion for $\mathcal R^{1/2}$.  The expansion of $R_i$ follows directly from the resolvent series for $(I+r\Sigma_i)^{-1}$.
\end{proof}

\section{Bai--Silverstein calibration and testing}\label{app:BS-calibration}
We record the Bai--Silverstein LSS formulas used in Corollary~\ref{cor:lss-clt}, with the sign and real/complex conventions needed for the calibration.  Throughout this appendix,
\[
m(z)=\int\frac{1}{x-z}\,\dd F(x),
\]
so $\Im m(z)>0$ for $z\in\mathbb C^+$.  For the identity population, the companion transform $\underline m_c$ satisfies
\begin{equation}\label{eq:BS-companion}
z=-\frac{1}{\underline m_c(z)}+\frac{c}{1+\underline m_c(z)},
\qquad \Im\underline m_c(z)>0,
\end{equation}
and $\underline m_c(z)=-(1-c)/z+c\,m_c(z)$.  The statistic in Corollary~\ref{cor:lss-clt} is centered by the Marchenko--Pastur law at aspect ratio $c_p$, denoted $F^{c_p}$, namely $p\int f\,\dd F^{c_p}$; the limiting mean and variance below are evaluated at $c=\lim c_p$.  Replacing the limiting functionals by their $c_p$ counterparts in the standardized statistic is justified by continuity in $c$ and Slutsky's theorem.

For the real Gaussian model, the excess fourth cumulant is zero, while the real ($\beta=1$) correction remains.  With the convention above, a convenient contour representation of the resulting limiting mean is
\begin{equation}\label{eq:BS-mean}
\mathfrak m_c^{\rm BS}(f)
=-\frac{1}{2\pi i}\oint_\Gamma
f(z)
\frac{c\,\underline m_c(z)^3(1+\underline m_c(z))^{-3}}
{\left[1-c\,\underline m_c(z)^2(1+\underline m_c(z))^{-2}\right]^2}
\,\dd z,
\end{equation}
and the covariance functional is
\begin{equation}\label{eq:BS-var}
\mathfrak v_c^{\rm BS}(f)
=-\frac{1}{2\pi^2}
\oint_{\Gamma_1}\oint_{\Gamma_2}
f(z_1)f(z_2)
\frac{\underline m_c'(z_1)\underline m_c'(z_2)}
{[\underline m_c(z_1)-\underline m_c(z_2)]^2}
\,\dd z_1\dd z_2.
\end{equation}
Here all contours are positively oriented and enclose the Marchenko--Pastur support $[a_c,b_c]$ without enclosing the origin.  In \eqref{eq:BS-var}, $\Gamma_1$ and $\Gamma_2$ are taken disjoint and nested, so the denominator is never evaluated on the diagonal.  These formulas are the real Gaussian identity specialization of the analytic LSS CLT in~\cite{BaiSilverstein2004,BaiSilverstein2010}; the stated contour choices also fix the orientation convention used by the numerical implementation.

An independent check of \eqref{eq:BS-mean}--\eqref{eq:BS-var}, together with the positivity argument, comes from the Marchenko--Pastur Joukowski parametrization
\begin{equation}\label{eq:joukowski-mp}
x_c(\theta)=1+c-2\sqrt c\cos\theta,
\qquad 0\le\theta\le2\pi.
\end{equation}
Define
\begin{equation}\label{eq:fourier-coeff}
\widehat f_k
=\frac{1}{2\pi}\int_0^{2\pi}f(x_c(\theta))e^{-ik\theta}\,\dd\theta.
\end{equation}
In the same real Gaussian normalization, the Joukowski representation gives the equivalent mean formula
\begin{equation}\label{eq:BS-joukowski-mean}
\mathfrak m_c^{\rm BS}(f)
=\frac{f(a_c)+f(b_c)}{4}
-\frac{1}{2\pi}\int_0^\pi f(x_c(\theta))\,\dd\theta,
\end{equation}
and the corresponding variance formula
\begin{equation}\label{eq:BS-joukowski-var}
\mathfrak v_c^{\rm BS}(f)
=2\sum_{k=1}^{\infty} k\,|\widehat f_k|^2.
\end{equation}
Equation~\eqref{eq:BS-joukowski-var} is nonnegative and vanishes only when $f\circ x_c$ is constant.  Because $x_c([0,\pi])=[a_c,b_c]$, every analytic nonconstant $f$ on the Marchenko--Pastur support has strictly positive variance, proving the final assertion of Corollary~\ref{cor:lss-clt}.

For $f=f_\tau$, the support satisfies $a_c>0$ because $c<1$.  The principal square root is analytic on a neighborhood of $[a_c,b_c]$, and the argument
\[
\sqrt{(1+r)(1+rz)}-r\sqrt z
\]
is positive on the real support and has no zeros on a sufficiently small complex neighborhood of it.  The principal logarithm therefore defines an analytic continuation of $f_\tau$ around the support, as required by the LSS CLT.  The supplementary code evaluates \eqref{eq:BS-mean}--\eqref{eq:BS-var} and \eqref{eq:BS-joukowski-mean}--\eqref{eq:BS-joukowski-var} independently; at $(c,\tau)=(1/2,2)$ the two calculations agree to the reported tolerance.

\subsection{Testing consequences}\label{app:one-sample-consequences}
For the identity-covariance test in Section~\ref{sec:test}, write
\[
T_{p,n}^{(\tau)}=\sum_{j=1}^p f_\tau(\lambda_j(S_p)),
\]
and
\[
Z_{p,n}^{(\tau)}=
\frac{T_{p,n}^{(\tau)}-p b_{\tau,c_p}-\mathfrak m_{c_p}^{\rm BS}(f_\tau)}
{\sqrt{\mathfrak v_{c_p}^{\rm BS}(f_\tau)}}.
\]
Corollary~\ref{cor:lss-clt} gives $Z_{p,n}^{(\tau)}\Rightarrow N(0,1)$ under the exact real-Gaussian identity null.

\begin{proposition}[Separated-alternative consistency]\label{prop:test-consistency}
Under the Gaussian sampling model, let $F^{\Sigma_p}\Rightarrow H$ with population eigenvalues uniformly bounded above and away from zero.  If
\[
\Delta_{\tau,c,H}
:=\int f_\tau\,dF_{c,H}-b_{\tau,c}>0,
\]
then the one-sided level-$\alpha$ rule $\mathbf 1\{Z_{p,n}^{(\tau)}>z_{1-\alpha}\}$ has rejection probability tending to one.
\end{proposition}

\begin{proof}
The generalized Marchenko--Pastur theorem gives
\[
\frac1pT_{p,n}^{(\tau)}\longrightarrow\int f_\tau\,dF_{c,H}
\qquad\text{almost surely}.
\]
Hence the numerator of $Z_{p,n}^{(\tau)}$ is
\[
p\Delta_{\tau,c,H}+o_{\mathbb P}(p),
\]
whereas the null mean correction and standard deviation are $O(1)$.  The standardized statistic therefore diverges to $+\infty$ in probability.
\end{proof}

\begin{remark}\label{rem:dense-local}
At the smaller population scale
\[
\Sigma_p=I+p^{-1/2}H_p,
\qquad
\normop{H_p}=O(1),
\qquad
p^{-1}\tr H_p^2\to\eta,
\]
the scalar expansion $f_\tau(1+h)=h^2/[4(1+2/\tau)]+O(h^3)$ gives
\[
\Acal_{\rm cov,\tau}(I,\Sigma_p)
\longrightarrow\frac{\eta}{4(1+2/\tau)}.
\]
This is a population-geometric statement.  Sample-level local power additionally depends on the $O(1)$ LSS mean shift under a contiguous covariance deformation and is not asserted here.
\end{remark}

The results in this subsection are consequences of the one-sample LSS benchmark and are not used in the two-sample free-probability arguments.  The representative finite-sample power comparison remains in the Supplementary Numerical Material.

\bibliographystyle{unsrt}
\bibliography{main}

@article{MarchenkoPastur1967,
  author = {V. A. Marchenko and L. A. Pastur},
  title = {{Distribution of eigenvalues for some sets of random matrices}},
  journal = {Math. USSR-Sb.},
  volume = {1},
  year = {1967},
  pages = {457--483}
}

@article{SilversteinChoi1995,
  author = {J. W. Silverstein and S. I. Choi},
  title = {{Analysis of the limiting spectral distribution of large dimensional random matrices}},
  journal = {J. Multivariate Anal.},
  volume = {54},
  number = {2},
  year = {1995},
  pages = {295--309},
  doi = {10.1006/jmva.1995.1058}
}

@article{BaiSilverstein2004,
  author = {Z. D. Bai and J. W. Silverstein},
  title = {{CLT for linear spectral statistics of large-dimensional sample covariance matrices}},
  journal = {Ann. Probab.},
  volume = {32},
  number = {1A},
  year = {2004},
  pages = {553--605}
}

@book{BaiSilverstein2010,
  author = {Z. D. Bai and J. W. Silverstein},
  title = {{Spectral Analysis of Large Dimensional Random Matrices}},
  edition = {2},
  publisher = {Springer},
  address = {New York},
  year = {2010}
}

@book{CouilletDebbah2011,
  author = {R. Couillet and M. Debbah},
  title = {{Random Matrix Methods for Wireless Communications}},
  publisher = {Cambridge University Press},
  address = {Cambridge},
  year = {2011}
}

@book{YaoZhengBai2015,
  author = {J. Yao and S. Zheng and Z. D. Bai},
  title = {{Large Sample Covariance Matrices and High-Dimensional Data Analysis}},
  publisher = {Cambridge University Press},
  address = {Cambridge},
  year = {2015}
}

@article{HachemLoubatonNajim2007,
  author = {W. Hachem and P. Loubaton and J. Najim},
  title = {{Deterministic equivalents for certain functionals of large random matrices}},
  journal = {Ann. Appl. Probab.},
  volume = {17},
  number = {3},
  year = {2007},
  pages = {875--930},
  doi = {10.1214/105051606000000925}
}

@article{ChenZhengZou2026,
  author = {W. Chen and S. Zheng and T. Zou},
  title = {{Spectral properties of high-dimensional rescaled sample correlation matrices}},
  journal = {Random Matrices: Theory Appl.},
  volume = {15},
  number = {2},
  year = {2026},
  pages = {2550030},
  doi = {10.1142/S2010326325500303}
}

@article{WangEtAl2026,
  author = {Q. Wang and R. Lin and X. Wang and J. Chen},
  title = {{An integrated test on the linear structure of high-dimensional covariance matrices}},
  journal = {Random Matrices: Theory Appl.},
  volume = {15},
  number = {1},
  year = {2026},
  pages = {2550027},
  doi = {10.1142/S2010326325500273}
}

@article{Gelbrich1990,
  author = {M. Gelbrich},
  title = {{On a formula for the $L^2$ Wasserstein metric between measures on Euclidean and Hilbert spaces}},
  journal = {Math. Nachr.},
  volume = {147},
  year = {1990},
  pages = {185--203}
}

@article{Takatsu2011,
  author = {A. Takatsu},
  title = {{Wasserstein geometry of Gaussian measures}},
  journal = {Osaka J. Math.},
  volume = {48},
  year = {2011},
  pages = {1005--1026}
}

@article{BhatiaJainLim2019,
  author = {R. Bhatia and T. Jain and Y. Lim},
  title = {{On the Bures--Wasserstein distance between positive definite matrices}},
  journal = {Expo. Math.},
  volume = {37},
  year = {2019},
  pages = {165--191}
}

@inproceedings{TiomokoCouillet2019,
  author = {M. Tiomoko and R. Couillet},
  title = {{Random matrix-improved estimation of the Wasserstein distance between two centered Gaussian distributions}},
  booktitle = {Proc. 27th Eur. Signal Process. Conf. (EUSIPCO)},
  year = {2019},
  pages = {1--5}
}

@article{YangZhang2026,
  author = {J. Yang and Y. Zhang},
  title = {{Closed forms for Gaussian Kullback--Leibler unbalanced optimal transport without coupling entropy}},
  journal = {arXiv preprint arXiv:2605.02497},
  year = {2026}
}

@phdthesis{JanatiThesis2021,
  author = {H. Janati},
  title = {{Advances in optimal transport and applications to neuroscience}},
  school = {Institut Polytechnique de Paris},
  year = {2021},
  note = {NNT: 2021IPPAG001}
}

@inproceedings{JanatiEtAl2020,
  author = {H. Janati and B. Muzellec and G. Peyr{\'e} and M. Cuturi},
  title = {{Entropic optimal transport between unbalanced Gaussian measures has a closed form}},
  booktitle = {Adv. Neural Inf. Process. Syst.},
  volume = {33},
  year = {2020}
}

@article{NakashimaEtAl2026,
  author = {H. Nakashima and S. Ganguly and K. Kashima},
  title = {{Globally solving unbalanced optimal transport and density control for Gaussian distributions}},
  journal = {arXiv preprint arXiv:2605.04246},
  year = {2026}
}

@article{LieroMielkeSavare2018,
  author = {M. Liero and A. Mielke and G. Savar{\'e}},
  title = {{Optimal entropy-transport problems and a new Hellinger--Kantorovich distance between positive measures}},
  journal = {Invent. Math.},
  volume = {211},
  year = {2018},
  pages = {969--1117}
}

@article{ChizatEtAl2018,
  author = {L. Chizat and G. Peyr{\'e} and B. Schmitzer and F.-X. Vialard},
  title = {{Scaling algorithms for unbalanced optimal transport problems}},
  journal = {Math. Comp.},
  volume = {87},
  number = {314},
  year = {2018},
  pages = {2563--2609},
  doi = {10.1090/mcom/3303}
}

@book{VoiculescuDykemaNica1992,
  author = {D. V. Voiculescu and K. J. Dykema and A. Nica},
  title = {{Free Random Variables}},
  series = {CRM Monograph Series},
  volume = {1},
  publisher = {American Mathematical Society},
  address = {Providence, RI},
  year = {1992}
}

@book{NicaSpeicher2006,
  author = {A. Nica and R. Speicher},
  title = {{Lectures on the Combinatorics of Free Probability}},
  series = {London Mathematical Society Lecture Note Series},
  volume = {335},
  publisher = {Cambridge University Press},
  address = {Cambridge},
  year = {2006}
}

@article{BelinschiAtoms2003,
  author = {S. T. Belinschi},
  title = {{The atoms of the free multiplicative convolution of two probability distributions}},
  journal = {Integral Equations Operator Theory},
  volume = {46},
  number = {4},
  year = {2003},
  pages = {377--386},
  doi = {10.1007/s00020-002-1145-4}
}

@article{BercoviciVoiculescu1993,
  author = {H. Bercovici and D. Voiculescu},
  title = {{Free convolution of measures with unbounded support}},
  journal = {Indiana Univ. Math. J.},
  volume = {42},
  number = {3},
  year = {1993},
  pages = {733--773},
  doi = {10.1512/iumj.1993.42.42033}
}

@article{BaiYin1993,
  author = {Z. D. Bai and Y. Q. Yin},
  title = {{Limit of the smallest eigenvalue of a large dimensional sample covariance matrix}},
  journal = {Ann. Probab.},
  volume = {21},
  number = {3},
  year = {1993},
  pages = {1275--1294},
  doi = {10.1214/aop/1176989118}
}

@article{NajimYao2016,
  author = {J. Najim and J. Yao},
  title = {{Gaussian fluctuations for linear spectral statistics of large random covariance matrices}},
  journal = {Ann. Appl. Probab.},
  volume = {26},
  number = {3},
  year = {2016},
  pages = {1837--1887},
  doi = {10.1214/15-AAP1135}
}

@article{RaoEdelman2008,
  author = {N. Raj Rao and A. Edelman},
  title = {{The polynomial method for random matrices}},
  journal = {Found. Comput. Math.},
  volume = {8},
  number = {6},
  year = {2008},
  pages = {649--702},
  doi = {10.1007/s10208-007-9013-x}
}

@article{Ji2021,
  author = {H. C. Ji},
  title = {{Regularity properties of free multiplicative convolution on the positive line}},
  journal = {Int. Math. Res. Not. IMRN},
  volume = {2021},
  number = {6},
  year = {2021},
  pages = {4522--4563},
  doi = {10.1093/imrn/rnaa152}
}

@article{Redelmeier2014,
  author = {C. E. I. Redelmeier},
  title = {{Real second-order freeness and the asymptotic real second-order freeness of several real matrix models}},
  journal = {Int. Math. Res. Not. IMRN},
  volume = {2014},
  number = {12},
  year = {2014},
  pages = {3353--3395}
}

@article{CollinsMale2014,
  author  = {Beno\^{\i}t Collins and Camille Male},
  title   = {The strong asymptotic freeness of {Haar} and deterministic matrices},
  journal = {Annales Scientifiques de l'\'Ecole Normale Sup\'erieure},
  volume  = {47},
  number  = {1},
  pages   = {147--163},
  year    = {2014},
  doi     = {10.24033/asens.2211}
}

@article{Biane1998,
  author = {P. Biane},
  title = {{Processes with free increments}},
  journal = {Math. Z.},
  volume = {227},
  number = {1},
  year = {1998},
  pages = {143--174},
  doi = {10.1007/PL00004363}
}

@article{BelinschiBercovici2007,
  author = {S. T. Belinschi and H. Bercovici},
  title = {{A new approach to subordination results in free probability}},
  journal = {J. Anal. Math.},
  volume = {101},
  year = {2007},
  pages = {357--365},
  doi = {10.1007/s11854-007-0013-1}
}

@article{ChizatDynamic2018,
  author = {L. Chizat and G. Peyr{\'e} and B. Schmitzer and F.-X. Vialard},
  title = {{Unbalanced optimal transport: Dynamic and Kantorovich formulations}},
  journal = {J. Funct. Anal.},
  volume = {274},
  number = {11},
  year = {2018},
  pages = {3090--3123},
  doi = {10.1016/j.jfa.2018.03.008}
}

@incollection{SejournePeyreVialard2023,
  author = {T. S{\'e}journ{\'e} and G. Peyr{\'e} and F.-X. Vialard},
  title = {{Unbalanced optimal transport, from theory to numerics}},
  booktitle = {Handbook of Numerical Analysis},
  volume = {24},
  publisher = {Elsevier},
  year = {2023},
  pages = {407--471},
  doi = {10.1016/bs.hna.2022.11.003}
}

@article{ZhengBaiYao2015,
  author = {S. Zheng and Z. D. Bai and J. Yao},
  title = {{Substitution principle for CLT of linear spectral statistics of high-dimensional sample covariance matrices with applications to hypothesis testing}},
  journal = {Ann. Statist.},
  volume = {43},
  number = {2},
  year = {2015},
  pages = {546--591},
  doi = {10.1214/14-AOS1292}
}

@article{CortinovisYing2025,
  author = {A. Cortinovis and L. Ying},
  title = {{Computing free convolutions via contour integrals}},
  journal = {Random Matrices Theory Appl.},
  volume = {14},
  number = {1},
  year = {2025},
  pages = {2450024},
  doi = {10.1142/S2010326324500242}
}

@article{SpeicherVargas2012,
  author = {R. Speicher and C. Vargas},
  title = {{Free deterministic equivalents, rectangular random matrix models, and operator-valued free probability theory}},
  journal = {Random Matrices Theory Appl.},
  volume = {1},
  number = {2},
  year = {2012},
  pages = {1150008},
  doi = {10.1142/S2010326311500080}
}

@article{CouilletTiomokoZozorMoisan2019,
  author = {R. Couillet and M. Tiomoko and S. Zozor and E. Moisan},
  title = {{Random matrix-improved estimation of covariance matrix distances}},
  journal = {J. Multivar. Anal.},
  volume = {174},
  year = {2019},
  pages = {104531},
  doi = {10.1016/j.jmva.2019.06.009}
}

@article{PereiraMestreGregoratti2024,
  author = {R. Pereira and X. Mestre and D. Gregoratti},
  title = {{Asymptotics of distances between sample covariance matrices}},
  journal = {IEEE Trans. Signal Process.},
  volume = {72},
  year = {2024},
  pages = {1460--1474},
  doi = {10.1109/TSP.2024.3368771}
}

@inproceedings{BouchardEtAl2024,
  author = {F. Bouchard and A. Mian and M. Tiomoko and G. Ginolhac and F. Pascal},
  title = {{Random matrix theory improved Fr{\'e}chet mean of symmetric positive definite matrices}},
  booktitle = {Proceedings of the 41st International Conference on Machine Learning},
  series = {Proceedings of Machine Learning Research},
  volume = {235},
  pages = {4403--4415},
  year = {2024},
  publisher = {PMLR}
}

@inproceedings{VacherVialard2023,
  author = {A. Vacher and F.-X. Vialard},
  title = {{Semi-dual unbalanced quadratic optimal transport: Fast statistical rates and convergent algorithm}},
  booktitle = {Proceedings of the 40th International Conference on Machine Learning},
  series = {Proceedings of Machine Learning Research},
  volume = {202},
  pages = {34734--34758},
  year = {2023},
  publisher = {PMLR}
}

@article{GallouetGhezziVialard2025,
  author = {T. Gallou{\"e}t and R. Ghezzi and F.-X. Vialard},
  title = {{Regularity theory and geometry of unbalanced optimal transport}},
  journal = {J. Funct. Anal.},
  volume = {289},
  number = {7},
  year = {2025},
  pages = {111042},
  doi = {10.1016/j.jfa.2025.111042}
}

@article{PonnopratIsobeImaizumi2026,
  author = {D. Ponnoprat and N. Isobe and M. Imaizumi},
  title = {{Minimax optimal estimation of transport-growth pairs in unbalanced optimal transport}},
  journal = {arXiv preprint arXiv:2605.08705},
  year = {2026}
}

@article{LodhiaLevinLevina2022,
  author  = {Asad Lodhia and Keith Levin and Elizaveta Levina},
  title   = {Matrix means and a novel high-dimensional shrinkage phenomenon},
  journal = {Bernoulli},
  volume  = {28},
  number  = {4},
  pages   = {2578--2605},
  year    = {2022},
  doi     = {10.3150/21-BEJ1430}
}

@article{CollinsHayase2023,
  author = {Beno{\^\i}t Collins and Tomohiro Hayase},
  title = {{Asymptotic freeness of layerwise Jacobians caused by invariance of multilayer perceptron: the Haar orthogonal case}},
  journal = {Commun. Math. Phys.},
  volume = {397},
  number = {1},
  year = {2023},
  pages = {85--109},
  doi = {10.1007/s00220-022-04441-7}
}

\end{document}